\documentclass{article}

\usepackage{epsfig,xspace}
\usepackage{latexsym}
\usepackage[matrix,arrow,curve]{xy}\xyoption{all}

\usepackage[usenames,dvipsnames,svgnames,table]{xcolor}
\usepackage{amssymb,amsmath,latexsym,amsfonts,amsthm,alltt}
\usepackage[usenames,dvipsnames,svgnames,table]{xcolor}
\usepackage[dvips,usenames]{stmaryrd}
\usepackage{url}
\usepackage{graphicx}
\usepackage{float}
\usepackage{amstext}
\usepackage{cite}
\usepackage{hyperref}
\usepackage{MnSymbol}
\usepackage{accents}

\newtheorem{defn}{Definition}[section]
\newtheorem{rem}[defn]{Remark}
\newtheorem{thm}[defn]{Theorem}
\newtheorem{lemma}[defn]{Lemma}
\newtheorem{prop}[defn]{Proposition}
\newtheorem{coro}[defn]{Corollary}

\newtheorem{ex}{Example}[section]

\newcommand\crule[3][black]{\textcolor{#1}{\rule{#2}{#3}}}

\newcommand{\ra}{\rightarrow}
\newcommand{\lra}{\longrightarrow}

\newcommand{\Ra}{\Rightarrow}

\newcommand{\midsp}{\;|\;}

\newcommand{\telos}{\hfill$\Box$}
\newcommand{\type}[1]{{\tt #1}}

\newcommand{\iso}{\backsimeq}

\newcommand{\upv}{\upVdash}
\newcommand{\rperp}{\mbox{${}^{\upv}$}}

\newcommand{\gphi}{{\mathcal  G}(Y)}
\newcommand{\gpsi}{{\mathcal  G}(X)}

\newcommand{\bbox}{\blacksquare}

\newcommand{\lperp}{{}\rperp}

\newcommand{\lbdiamond}{\raisebox{-4pt}{\mbox{\Huge {$\filleddiamond$}}}}

\newcommand{\lbbox}{\raisebox{-1.2pt}[0pt][0pt]{\crule[black]{0.27cm}{0.27cm}}\hspace*{1pt}}

\newcommand{\filt}{\mbox{\rm Filt}}

\newcommand{\idl}{\mbox{\rm Idl}}

\title{Duality for Normal Lattice Expansions and Sorted, Residuated Frames with Relations}
\author{Chrysafis (Takis) Hartonas\\
University of Thessaly, Greece\\
$\type{hartonas@uth.gr}$}

\begin{document}
\maketitle

\begin{abstract}
We revisit the problem of Stone duality for lattices with various quasioperators, first studied in \cite{sdl-exp}, presenting a fresh duality result. The new result is an improvement over that of \cite{sdl-exp} in two important respects. First, the axiomatization of frames in \cite{sdl-exp} is now simplified, partly by incorporating Gehrke's proposal \cite{mai-gen} of section stability for relations. Second, morphisms are redefined so as to preserve Galois stable (and co-stable) sets and we rely for this, partly again, on Goldblatt's \cite{goldblatt-morphisms2019} recently proposed definition of bounded morphisms for polarities.

In studying the dual algebraic structures associated to polarities with relations we demonstrate that stable/co-stable set operators result as the Galois closure of the restriction of classical (though sorted) image operators generated by the frame relations to Galois stable/co-stable sets. This provides a proof, at the representation level, that non-distributive logics can be viewed as fragments of sorted, residuated (poly)modal logics, a research direction initiated in \cite{pll7,redm}.
\end{abstract}
\vspace*{2mm}
\noindent
{\bf Keywords:}  Non-distributive logics, lattices with quasioperators, normal lattice expansions, Stone duality, polarities, polarities with relations

\section{Introduction}
\label{intro}
In this article we address and resolve the problem of Stone duality for the category of lattices with quasioperators. The work presented here is a significant improvement over our own \cite{sdl} (with Dunn) and \cite{sdl-exp}. The axiomatization of frames and frame relations in \cite{sdl-exp} was rather cumbersome and it is now simplified, partly by incorporating Gehrke's proposal \cite{mai-gen} of section stability for relations. Morphisms in the category of frames (polarities with relations) dual to lattices with quasioperators are defined so that not only clopens in the dual complex algebra of the frame be preserved (as in \cite{sdl,sdl-exp}) but so that the inverses of frame morphisms are morphisms of the complete lattices of stable sets dual to frames (their complex algebras), preserving arbitrary joins and meets. Polarity morphisms, as we define them, are the same as Goldblatt's bounded morphisms for polarities \cite{goldblatt-morphisms2019}, but we diverge from \cite{goldblatt-morphisms2019} in the extension to morphisms for polarities with relations. We also diverge from \cite{mai-gen,goldblatt-morphisms2019} in the definition of canonical relations and in the way operators are defined in frames from relations. We do this with the express objective of demonstrating that quasi-operators in the complex algebra of stable sets of a frame can be obtained as the Galois closure of the restriction to Galois sets of the classical though sorted image operators generated from the relations, in the J\'{o}nsson-Tarski style \cite{jt1}. The logical significance of this is that it demonstrates at the representation level that non-distributive logics are fragments of sorted residuated polymodal logics (their modal companions).

The structure of this article is as follows. Section \ref{normal lat exp section} is an introductory section, defining the category {\bf NLE}$_\tau$ of normal lattice expansions of some similarity type $\tau$ (a sequence of distribution types of lattice operators). 

In Section \ref{frames section} we present basic definitions and results for polarities (equivalently, sorted residuated frames). In Remark \ref{notation rem} we carefully list all our notational conventions, to ensure the reader has an effortless and seamless understanding of our notation. 

With Section  \ref{frames with relations section} we address the first issue of significance for our purposes, which is to define operators from relations, properly axiomatized, so as to ensure complete distribution properties of the defined operators. In the same section we list the first four axioms for the objects of the category {\bf SRF}$_\tau$ of sorted residuated frames (polarities) with relations. 

Section \ref{bounded section} turns to studying morphisms, first for frames in the absence of additional relations (Section \ref{polarity morphisms section}) and then for frames with additional relations of sort types in some similarity type $\tau$ (Section  \ref{morphisms for relations section}). Further axioms for frame morphisms are stated and the category {\bf SRF}$_\tau$ defined (Definition \ref{tau-frames defn}). 

Section \ref{representation section} defines a contravariant functor from the category {\bf NLE}$_\tau$ to the category {\bf SRF}$_\tau$. For lattice representation, we rely on \cite{sdl} and for the representation of normal lattice operators we draw on \cite{dloa} and \cite{sdl-exp} and we review the canonical frame construction in Section \ref{canonical frame section}.  Sections \ref{canonical-properties} and \ref{homos section} are devoted to proving that the frame axioms hold for the dual frame of a normal lattice expansion and for the duals of normal lattice expansion homomorphisms. 

Stone duality is addressed in Section \ref{stone section}. To ensure a Stone duality theorem can be proven we extend the axiomatization for frames (defining a smaller category ${\bf SRF}^*_\tau$), drawing on \cite{sdl,sdl-exp}, we verify that all additional axioms hold for the canonical frame construction and conclude with a Stone duality theorem.

In the Conclusions section \ref{conc} we summarize the results obtained and sketch a potential area for further research.

\section{Normal Lattice Expansions (NLE's)}
\label{normal lat exp section}
In this section, we introduce the category of normal lattice expansions of some given similarity type.

Let $\{1,\partial\}$ be a 2-element set, $\mathcal{L}^1=\mathcal{L}$ and $\mathcal{L}^\partial=\mathcal{L}^{op}$ (the opposite lattice). Extending established terminology \cite{jt1}, a function $f:\mathcal{L}_1\times\cdots\times\mathcal{L}_n\lra\mathcal{L}_{n+1}$ is {\em additive}, or a {\em normal operator}, if it distributes over finite joins of the lattice $\mathcal{L}_i$, for each $i=1,\ldots n$, delivering a join in $\mathcal{L}_{n+1}$.

\begin{defn}\rm
\label{normal lattice op defn}
An $n$-ary operation $f$ on a bounded lattice $\mathcal L$ is {\em a normal lattice operator of distribution type  $\delta(f)=(i_1,\ldots,i_n;i_{n+1})\in\{1,\partial\}^{n+1}$}  if it is a normal additive function  $f:{\mathcal L}^{i_1}\times\cdots\times{\mathcal L}^{i_n}\lra{\mathcal L}^{i_{n+1}}$ (distributing over finite joins in each argument place), where  each $i_j$, for  $j=1,\ldots,n+1$,   is in the set $\{1,\partial\}$, hence ${\mathcal L}^{i_j}$ is either $\mathcal L$, or ${\mathcal L}^\partial$.

If $\tau$ is a tuple (sequence) of distribution types, a {\em normal lattice expansion of (similarity) type $\tau$} is a lattice with a normal lattice operator of distribution type $\delta$ for each $\delta$ in $\tau$.
\end{defn}
Normal lattice operators in the above sense are sometimes referred to in the literature as {\em quasioperators}.

\begin{defn}\rm
The {\em category {\bf NLE}$_\tau$}, for a fixed similarity type $\tau$, has normal lattice expansions of type $\tau$ as objects. Its morphisms are the usual algebraic homomorphisms.
\end{defn}

\begin{ex}\rm
A modal normal diamond operator $\Diamond$ is a normal lattice operator of distribution type $\delta(\Diamond)=(1;1)$, i.e. $\Diamond:\mathcal{L}\lra\mathcal{L}$, distributing over finite joins of $\mathcal{L}$. A normal box operator $\Box$ is also a normal lattice operator in the sense of Definition \ref{normal lattice op defn}, of distribution type $\delta(\Box)=(\partial;\partial)$, i.e. $\Box:\mathcal{L}^\partial\lra\mathcal{L}^\partial$ distributes over finite joins of $\mathcal{L}^\partial$, which are then just meets of $\mathcal{L}$.

  An {\bf FL}$_{ew}$-algebra (also referred to as a full {\bf BCK}-algebra, or a commutative integral residuated lattice) $\mathcal{A}=(L,\wedge,\vee,0,1,\circ,\ra)$ is a normal lattice expansion, where $\delta(\circ)=(1,1;1)$, $\delta(\ra)=(1,\partial;\partial)$, i.e. $\circ:\mathcal{L}\times\mathcal{L}\lra\mathcal{L}$ and $\ra\;:\mathcal{L}\times\mathcal{L}^\partial\lra\mathcal{L}^\partial$ are both normal lattice operators with the familiar distribution properties.

  The Grishin operators \cite{grishin} $\leftharpoondown,\star,\rightharpoondown$, satisfying the familiar co-residuation conditions $a\geq c\leftharpoondown b$ iff $a\star b\geq c$ iff $b\geq a\rightharpoondown c$ have the respective distribution properties, which are exactly captured by assigning to them the distribution types $\delta(\star)=(\partial,\partial;\partial)$ ($\star$ behaves like a binary box operator), $\delta(\leftharpoondown)=(1,\partial;1)$ and $\delta(\rightharpoondown)=(\partial,1;1)$.
\end{ex}
Distributive normal lattice expansions are the special case where the underlying lattice is distributive. BAO's (Boolean Algebras with Operators) \cite{jt1,jt2} are the special case where the underlying lattice is a Boolean algebra and all normal operators distribute over finite joins of the Boolean algebra, i.e. they are all of distribution types of the form $(1,\ldots,1;1)$. For BAO's,  operators of other distribution types can be obtained by composition of operators with Boolean complementation. For example, in studying residuated Boolean algebras \cite{residBA}, J\'{o}nsson and Tsinakis introduce a notion of {\em conjugate operators} and they show that intensional implications (division operations) $\backslash,/$ (the residuals of the product operator $\circ$)  are interdefinable with the conjugates (at each argument place) $\lhd,\rhd$ of $\circ$, i.e. $a\backslash b=(a\rhd b^-)^-$ and $a\rhd b=(a\backslash b^-)^-$ (and similarly for $/$ and $\lhd$, see \cite{residBA} for details). Note that $\backslash, /$ are not operators, whereas $\lhd,\rhd$ are.

The relational representation of BAO's in \cite{jt1}, extending Stone's representation \cite{stone1} of Boolean algebras using the space of ultrafilters of the algebra, forms the technical basis of the subsequently introduced by Kripke 
possible worlds semantics, with its well-known impact on the development of normal modal logics.
This has been extended to the case of the logics of distributive lattices with various quasioperators, see \cite{dunn-ggl} and \cite{Sofronie-Stokkermans00,Sofronie-Stokkermans00a}, for example, now based on the Priestley representation \cite{hilary} of distributive lattices in ordered Stone spaces (simplifying Stone's original representation \cite{stone2} of distributive lattices), using the space of prime filters.

For non-distributive lattices, a number of different representation results have been published and we have briefly reviewed the area in \cite{discr}. For the reader's convenience, the review is presented below as well.

Dating back in 1978, Alasdair Urquhart's representation of bounded lattices \cite{urq} was conceived with the express objective of reducing to the Priestley representation \cite{hilary} of distributive lattices, when the represented lattice is distributive. This was achieved by letting the points of the dual lattice frame be maximal disjoint filter-ideal pairs, which reduce to pairs of a prime filter and its complement (a prime ideal) in the distributive case. Coordinatewise ordering endows the dual space with two orderings, $(x,y)\leq_1(x',y')$ iff $x\subseteq x'$ and $(x,y)\leq_2(x',y')$ iff $y\subseteq y'$. In any doubly-ordered space $(U,\leq_1,\leq_2)$, where $\leq_i$ is a pre-order and it is assumed that $u\leq_1 u'$ and $u\leq_2 u'$ implies that $u=u'$, if $U^{\uparrow_i}\;(i=1,2)$ are the families of sets that are $\leq_1$-increasing and $\leq_2$-increasing, respectively, then the maps $r:(U^{\uparrow_2})^{op}\leftrightarrows U^{\uparrow_1}:\ell$ defined by
\begin{eqnarray*}
\ell D &=&\{u\in U\midsp \forall u'\in U\;(u\leq_1 u'\;\lra\; u'\not\in D)\}\\
rC &=&\{u\in U\midsp\forall u'\in U\;(u\leq_2 u'\;\lra\; u'\not\in C)\}
\end{eqnarray*}
form a Galois connection, $C\subseteq\ell D$ iff $D\subseteq rC$. A set $C$ is $\ell$-stable iff $C=\ell rC$ and similarly for $r$-stable sets. If the lattice $\mathcal{L}$ is distributive, then $\ell,r$ are both the set-complement operators. With
representation map defined by $\alpha(a)=\{(x,y)\midsp a\in x\}$,  Urquhart topologized the space via the subbasis $\{-\alpha(a)\midsp a\in\mathcal{L}\}$ $\cup\{-r\alpha(a')\midsp a'\in\mathcal{L}\}$ and characterized the image of the representation function as the collection of
all doubly-closed (both $A$ and $rA$ closed in the topology) $\ell$-stable sets. Urquhart's representation does not extend to a full categorical duality.

Allwein and Hartonas (1993) \cite{iulg-bounded} modified Urquhart's representation by merely dropping the maximality constraint, working with all disjoint filter-ideal pairs. This allowed for extending in \cite{iulg-bounded} the representation to a full categorical duality. While Urquhart's L-spaces become Priestley's  compact, totally order-disconnected spaces in the distributive case, this is no longer the case with the duality of \cite{iulg-bounded}. It is shown, however, that in the distributive case the EL-spaces of \cite{iulg-bounded} contain the Priestley space of the lattice, as a distinguished subspace and conditions are provided under which the dual space of a lattice is a space arising from a distributive lattice. The duality of \cite{iulg-bounded} went unnoticed for quite some time, till Craig, Haviar and Priestley (2013) \cite{craig2} revisited it and, combining with Plo\v{s}\v{c}ica's  lattice representation \cite{plo}, gave a new lattice representation and duality, in the spirit of natural dualities \cite{natural-duality}.

Miroslav Plo\v{s}\v{c}ica (1995) \cite{plo} observed that, given a maximal filter-ideal pair $(x,y)$ of a lattice $\mathcal{L}$, the partial map $f:\mathcal{L}\lra\overline{{\bf 2}}$ (where $\overline{{\bf 2}}$ designates the two-element lattice $\{0,1\}$)  defined by $f(a)=\left\{\begin{array}{cl} 1 & \mbox{ if } a\in x\\ 0 & \mbox{ if }a\in y\\ \mbox{undefined} & \mbox{ otherwise}\end{array}\right.$ is a maximal partial homomorphism (mph) and that for any mph $f:\mathcal{L}\lra\overline{{\bf 2}}$, the pair $(f^{-1}(1), f^{-1}(0))$ is a maximal filter-ideal pair. In addition, the two preorders $\leq_1,\leq_2$ of an Urquhart frame $(U,\leq_1,\leq_2)$ are re-captured in the space of maximal partial homomorphisms by setting $f\leq_1 g$ iff $f^{-1}(1)\subseteq g^{-1}(1)$ and $f\leq_2 g$ iff $f^{-1}(0)\subseteq g^{-1}(0)$.  Building on this correspondence of maximal partial homomorphisms and maximal filter-ideal pairs, Plo\v{s}\v{c}ica presented an alternative lattice representation, by recasting the Urquhart representation in terms of maximal partial homomorphisms. The dual spaces of lattices are defined to be graphs with a topology $\tau$ on the carrier set $X$, $(X,E,\tau)$, where $E\subseteq X\times X$ is a reflexive binary relation, and morphisms of graphs with a topology are defined as the continuous graph homomorphisms. A partial map between topologized graphs is a partial morphism if its domain is a closed set (in the respective topology) and its restriction to its domain is a morphism of topologized graphs.  Letting $D(\mathcal{L})=(\mathcal{L}^{mp}(\mathcal{L},\overline{{\bf 2}}),E)$ be the set $\mathcal{L}^{mp}(\mathcal{L},\overline{{\bf 2}})$ of mph's  $f:\mathcal{L}\lra\overline{{\bf 2}}$ equipped with an ordering relation $fEg$ iff $f(a)\leq g(a)$ for all $a\in{\rm dom}(f)\cap{\rm dom}(g)$ and topologized with the subbasis of all sets of the form $A_a=\{f\in D(\mathcal{L})\midsp f(a)=0\}, B_a=\{f\in D(\mathcal{L})\midsp f(a)=1\}$, for each lattice element $a$,  letting also $\tilde{{\bf 2}}$ designate the set $\{0,1\}$ with the discrete topology and the obvious ordering $0<1$, Plo\v{s}\v{c}ica showed that the lattice $\mathcal{L}$ is isomorphic to the lattice of maximal partial morphisms of topologized graphs $\phi:D(\mathcal{L})\lra\tilde{{\bf 2}}$, ordered by the rule $\phi\leq \psi$ iff $\phi^{-1}(1)\subseteq \psi^{-1}(1)$, where the isomorphism relies on the fact that is verified in \cite{plo} that every maximal partial morphism  $e:D(\mathcal{L})\lra\tilde{{\bf 2}}$ is an evaluation map $e_a$, for some $a\in\mathcal{L}$, $e_a(f)=1$ iff $f(a)=1$. A discussion of extending the representation to a full duality is not provided in \cite{plo}.

Hartung (1992) \cite{hartung}, motivated by Wille's formal concept analysis (FCA) \cite{wille2} represented lattices as sublattices of the formal concept lattice of a clarified and reduced formal context $(X,I,Y)$. The context relation $I$ of a formal context $(X,I,Y)$ generates a Galois connection on subsets of $X$ and $Y$, typically designated in FCA using a priming notation: $U'=\{y\in Y\midsp\forall x\in X\;(x\in U\lra xIy)\}$, for $U\subseteq X$, and $V'=\{x\in X\midsp \forall y\in Y\;(y\in V\lra xIy)\}$, for $V\subseteq Y$.  Hartung's dual lattice spaces are the topological formal contexts $((X,\rho), (Y,\sigma),I)$ such that (i) if $U\subseteq X$ is closed in the topology $\rho$, then so is its Galois-closure $U''$ and similarly for $V\subseteq Y$ and (ii) the set of topologically closed and Galois stable $A=A''$ subsets $A\subseteq X$ such that $A'$ is closed in the topology $\sigma$ on $Y$ is a subbasis of closed sets for $\rho$ (i.e. $\rho$ is generated by taking arbitrary intersections of finite unions of subbasis elements) and similarly for the topologically closed and Galois stable $B=B''$ subsets $B\subseteq Y$. The set of closed concepts of a formal context is defined to contain the concepts $(A,B)$ such that $A,B$ are closed sets in the respective topologies. In the canonical formal context $(X,I,Y)$, the sets $X=U_1,Y=U_2$ are in fact the first and second projections of Urquhart's space $U$ of maximally disjoint filter-ideal pairs and $I\subseteq X\times Y$ is defined by $xIy$ iff $x\cap y\neq\emptyset$. A lattice element $a$ is represented as the formal concept $(A,B)$, where $A=\{x\in X\midsp a\in x\}$ and $B=A'$ and the lattice $\mathcal{L}$ is identified as the lattice of closed concepts of the canonical formal context of $\mathcal{L}$. Hartung's representation, too, does not extend to a full duality.

Hartonas-Dunn (1993) \cite{iulg}, published in (1997) \cite{sdl}, provide a lattice representation and duality result that is based on the representation of semilattices and of Galois connections of \cite{iulg} and abstracts on Goldblatt's \cite{goldb} representation of ortholattices, replacing orthocomplementation with the trivial Galois connection (the identity map) $\imath:\mathcal{L}_\wedge\leftrightarrows\mathcal{L}_\vee^{op}$ between the meet sub-semilattice and the join sub-semilattice of a lattice $\mathcal{L}$. Galois connections between meet semilattices (or merely partial orders) $\lambda:\mathcal{S}_1\leftrightarrows\mathcal{S}_2^{op}:\rho$  are represented in \cite{iulg} as binary relations $R\subseteq\filt(\mathcal{S}_1)\times\filt(\mathcal{S}_2)$ defined by $Rxy$ iff $\exists a\in x\;\lambda a\in y$, which then specializes in the case of the trivial Galois connection (in fact, duality) $\imath:\mathcal{L}_\wedge\leftrightarrows\mathcal{L}_\vee^{op}$ to the relation $\upv\;\subseteq \filt(\mathcal{L})\times\idl(\mathcal{L})$ defined by $x\upv y$ iff $x\cap y\neq\emptyset$. The dual objects of bounded lattices are sorted, partially-ordered topological frames $(X,\upv,Y)$, with Stone topologies induced on each of the partial orders $X,Y$ by subbases $X^*,Y^*$ of subsets $C_j\subseteq X, D_j\subseteq Y, j\in J$ and their complements, for some indexing set $J$ such that (i) each of $X^*, Y^*$ is closed under binary intersections, (ii) the family of principal upper sets $\Gamma x=\{z\in X\midsp x\leq z\}$, with $x\in X$, is the intersection closure of $X^*$, i.e.  $\{\bigcap_{t\in J'}C_t\midsp J'\subseteq J\}=\{\Gamma x\midsp x\in X\}$ and similarly for $Y$. This is equivalent to the properties (a) $X,Y$ are complete lattices (b) each $C_j\in X^*$ is a principal upper set $\Gamma x_j$, for some $x_j\in X$ (and similarly for $Y$) and (c) the set $\{x_j\midsp\Gamma x_j\in X^*\}$ join-generates the complete lattice $X$ (and similarly for $Y$). Lattice elements are represented as the Galois stable sets $\{x\in X\midsp a\in x\}=\Gamma x_a$, where $x_a$ designates a principal filter, and co-represented as the co-stable sets $\{y\in Y\midsp a\in y\}=\Gamma y_a$, where $y_a$ designates a principal ideal and a lattice $\mathcal{L}$ is identified as the lattice of Galois stable, compact-open (clopen) subsets of $X$. Defining frame morphism by the requirement that their inverses preserve clopen sets, the representation is extended in \cite{sdl} to a full categorical duality.

Canonical extensions for bounded lattices were introduced by Gehrke and Harding \cite{mai-harding}, extending to the non-distributive case the notion of a perfect extension of Boolean algebras of J\'{o}nsson and Tarski \cite{jt1,jt2} and of a canonical extension of distributive lattices of J\'{o}nsson and Gehrke \cite{mai-jons}. A canonical lattice extension is defined as a pair $(\alpha,C)$ where $C$ is a complete lattice and $\alpha$ is an embedding of a lattice $\mathcal{L}$ into $C$ such that the following density and compactness requirements are satisfied
\begin{quote}
$\bullet$ (density) $\alpha[{\mathcal L}]$ is {\em dense} in $C$, where the latter means that every element of $C$ can be expressed both as a meet of joins and as a join of meets of elements in $\alpha[{\mathcal L}]$\\[0.5mm]
$\bullet$ (compactness) for any set $A$ of closed elements and any set  $B$ of open elements of $C$, $\bigwedge A\leq\bigvee B$ iff there exist finite subcollections $A'\subseteq A, B'\subseteq B$ such that $\bigwedge A'\leq\bigvee B'$
\end{quote}
where the {\em closed elements} of $C$ are defined in \cite{mai-harding} as the elements in the meet-closure of the representation map $\alpha$ and the {\em open elements} of $C$ are defined dually as the join-closure of the image of $\alpha$. Existence of canonical extensions is proven in \cite{mai-harding} by demonstrating that the compactness and density requirements are satisfied in the representation due to Hartonas and Dunn \cite{sdl}.

Andrew Craig in his PhD dissertation \cite{craig-phd} and in joint work with Haviar and Priestley \cite{craig2} and with Gouveia \cite{craig} revisited Plo\v{s}\v{c}ica's representation and, by combining with the Allwein-Hartonas representation \cite{iulg-bounded}, they showed how a canonical lattice extension can be derived from Plo\v{s}\v{c}ica's representation, with an extension to a functorial duality (in the spirit of natural dualities \cite{natural-duality}). For the first task,  the untopologized graph structure $(D(\mathcal{L}),E)$, designated by $D^\flat(\mathcal{L})$ in \cite{craig-phd}, is considered and it is shown that for any graph structure ${\bf X}=(X,E)$, the set $\mathcal{G}^{mp}({\bf X}, \underaccent{\tilde}{{\bf 2}})$ of partial graph morphisms into the two element graph ${0<1}$ is a complete lattice and, in particular, $\mathcal{G}^{mp}(D^\flat(\mathcal{L}),\underaccent{\tilde}{{\bf 2}})$ is a canonical extension of the lattice $\mathcal{L}$. This canonical extension construction, however, does not immediately lend itself as a basis for a full functorial duality.
For the task of extending to a duality, it is observed in \cite{craig,craig2,craig-phd} that whereas Plo\v{s}\v{c}ica's maximal partial homomorphisms are in bijective correspondence with Urquhart's \cite{urq} maximal disjoint filter-ideal pairs, the disjoint filter-ideal pairs of Allwein and Hartonas \cite{iulg-bounded} are in bijective correspondence with the set $\mathcal{L}^{sp}(\mathcal{L},\underline{{\bf 2}})$ of partial lattice homomorphisms $f:\mathcal{L}\lra\underline{{\bf 2}}$\footnote{Note that the two element lattice is designated by Plo\v{s}\v{c}ica as $\overline{{\bf 2}}$, whereas Craig, Haviar and Priestley write $\underline{{\bf 2}}$.} for which the pair $(f^{-1}(1),f^{-1}(0))$ is a disjoint filter-ideal pair. The Plo\v{s}\v{c}ica representation and the canonical extension construction are then recast as needed and it is shown that a categorical duality is also provable in this combined setting.

A Stone type duality for normal lattice expansions however has only first been presented in \cite{sdl-exp}. Part of the difficulty was in defining an appropriate notion of morphism and Goldblatt \cite{goldblatt-morphisms2019} reviews the related attempts to this issue. In \cite{sdl-exp} this problem was somewhat side-stepped, by restricting the definition of frame morphisms to such that their inverses are homomorphisms of the sublattices of clopen elements of the complex algebras of the frames, as in the lattice duality of \cite{sdl}. Another source of difficulty has been to define operators from suitably axiomatized relations on frames, so that the framework can serve the semantics of logics without distribution as J\'{o}nsson and Tarski's BAO's \cite{jt1} have served to the semantics of modal logics. In \cite{sdl-exp} we proposed an axiomatization of frames and relations, though the axiomatization appears somewhat forced and we provide a significant improvement in the present article.

\section{Sorted Residuated Frames (SRF's)}
\label{frames section}
Regard $\{1,\partial\}$ as a set of sorts and let $Z=(Z_1,Z_\partial)$ be a sorted set.
Sorted, residuated frames $\mathfrak{F}=(Z,I)=(Z_1,I,Z_\partial)$ are triples consisting of nonempty sets $Z_1=X,Z_\partial=Y$ and a binary relation $I\subseteq X\times Y$. The relation $I$ generates residuated operators $\largediamond:\powerset(X)\leftrightarrows\powerset(Y):\lbbox$ ($U\subseteq \lbbox V$ iff $\largediamond U\subseteq V$), defined by
\[
\largediamond U=\{y\in Y\midsp\exists x\in X\;(xIy\wedge x\in D)\}\hskip3mm \lbbox V=\{x\in X\midsp\forall y\in Y\;(xIy\lra y\in V)\}
\]
The dual sorted, residuated modal algebra of the (sorted, residuated) frame $(X,I,Y)$  is the algebra $\largediamond:\powerset(X)\leftrightarrows\powerset(Y):\lbbox$. By residuation, the compositions $\lbbox\largediamond$ and $\largesquare\lbdiamond$ are closure operators on $\powerset(X)$ and $\powerset(Y)$, respectively.

For a sorted frame $(X,I,Y)$, the complement of the frame relation $I$ will be consistently designated by $\upv$ and referred to as the {\em Galois relation} of the frame. It generates a Galois connection $(\;)\rperp:\powerset(X)\leftrightarrows\powerset(Y)^\partial:\lperp(\;)$ ($V\subseteq U\rperp$ iff $U\subseteq\lperp V$)
\begin{eqnarray*}
U\rperp&=\{y\in Y\midsp\forall x\in U\; x\upv y\} &= \{y\in Y\midsp U\upv y\}\\
\lperp V&=\{x\in X\midsp \forall y\in V\;x\upv y\}&=\{x\in X\midsp x\upv V\}
\end{eqnarray*}

Triples $(X,R,Y), R\subseteq X\times Y$, where $R$ is treated as the Galois relation of the frame, are variously referred to in the literature as {\em polarities}, after Birkhoff \cite{birkhoff}, as {\em formal contexts}, in the Formal Concept Analysis (FCA) tradition \cite{wille2}, or as {\em object-attribute (categorization, classification, or information) systems} \cite{ewa-rfca,rough-vakarelov}, or as {\em generalized Kripke frames} \cite{mai-gen}, or as {\em polarity frames} in the bi-approximation semantics of \cite{Suzuki-polarity-frames}.

Note that the residuated and Galois connected operators generate the same closure operators, on $\powerset(X)$,  $\lbbox\largediamond U={}\rperp(U\rperp)$ and on $Y$, $\largesquare\lbdiamond V=({}\rperp V)\rperp$. This follows from the fact that ${}\rperp V=\lbbox(-V)$ and $U\rperp=\largesquare(-U)$.

\begin{prop}\rm
The discrete categories of polarities and sorted residuated frames are equivalent.\telos
\end{prop}
The equivalence allows us to move in our arguments and definitions back-and-forth between sorted residuated frames and polarities.
Indeed, for our purposes, both the residuated pairs $\lbbox,\largediamond$ and $\largesquare,\lbdiamond$, as well as the Galois connection will be involved in definitions and arguments and we do not differentiate between polarities $(X,\upv,Y)$ and their associated sorted residuated frames $(X,I,Y)$, with $I$ being the complement of $\upv$.

A subset $A\subseteq X$ will be called {\em stable} if $A=\lbbox\largediamond A={}\rperp(A\rperp)$. Similarly, a subset $B\subseteq Y$ will be called {\em co-stable} if $B=\largesquare\lbdiamond B=({}\rperp B)\rperp$. Stable and co-stable sets will be referred to as {\em Galois sets}, disambiguating to {\em Galois stable} or {\em Galois co-stable} when needed and as appropriate.

$\gpsi,\gphi$ designate the complete lattices of stable and co-stable sets, respectively. Note that the Galois connection restricts to a duality $(\;)\rperp:\gpsi\iso\gphi^\partial:{}\rperp(\;)$.

Preorder relations are induced on each of the sorts, by setting for $x,z\in X$, $x\preceq z$ iff $\{x\}\rperp\subseteq\{z\}\rperp$ and, similarly, for $y,v\in Y$, $y\preceq v$ iff ${}\rperp\{y\}\subseteq{}\rperp\{v\}$.
A (sorted) frame is called {\em separated} if the preorders $\preceq$ (on $X$ and on $Y$) are in fact partial orders $\leq$.

\begin{rem}[{\bf Notational Conventions}]\rm
\label{notation rem}
For a sorted relation $R\subseteq\prod_{j=1}^{j=n+1}Z_{i_j}$, where $i_j\in\{1,\partial\}$ for each $j$ (and thus $Z_{i_j}=X$ if $i_j=1$ and $Z_{i_j}=Y$ when $i_j=\partial$), we make the convention to regard it as a relation $R\subseteq Z_{i_{n+1}}\times\prod_{j=1}^{j=n}Z_{i_j}$, we agree to write its sort type as $\sigma(R)=(i_{n+1};i_1\cdots i_n)$ and for a tuple of points of suitable sort we write $uRu_1\cdots u_n$ for $(u,u_1,\ldots,u_n)\in R$. We often display the sort type as a superscript, as in $R^\sigma$. Thus, for example, $R^{\partial 1\partial}$ is a subset of  $Y\times(X\times Y)$. In writing then $yR^{\partial 1 \partial}xv$ it is understood that $x\in X=Z_1$ and $y,v\in Y=Z_\partial$. The sort superscript is understood as part of the name designation of the relation, so that, for example, $R^{111}, R^{\partial\partial 1}$ name two different relations.

We use $\Gamma$ to designate  upper closure  $\Gamma U=\{z\in X\midsp\exists x\in U\;x\preceq z\}$, for $U\subseteq X$, and similarly for $U\subseteq Y$. $U$ is {\em increasing} (an upset) iff $U=\Gamma U$. For a singleton set $\{x\}\subseteq X$ we write $\Gamma x$, rather than $\Gamma(\{x\})$ and similarly for $\{y\}\subseteq Y$.

We typically use the standard FCA \cite{wille2} priming notation for each of the two Galois maps ${}\rperp(\;),(\;)\rperp$. This allows for stating and proving results for each of $\gpsi,\gphi$ without either repeating definitions and proofs, or making constant appeals to duality. Thus for a Galois set $G$, $G'=G\rperp$, if $G\in\gpsi$ ($G$ is a Galois stable set), and otherwise $G'={}\rperp G$, if $G\in\gphi$ ($G$ is a Galois co-stable set).

For an element $u$ in either $X$ or $Y$ and a subset $W$, respectively of $Y$ or $X$, we write $u|W$, under a well-sorting assumption, to stand for either $u\upv W$ (which stands for $u\upv w$, for all $w\in W$), or $W\upv u$ (which stands for $w\upv u$, for all $w\in W$), where well-sorting means that either $u\in X, W\subseteq Y$, or $W\subseteq X$ and $u\in Y$, respectively. Similarly for the notation $u|v$, where $u,v$ are elements of different sort.

We designate $n$-tuples (of sets, or elements) using a vectorial notation, setting  $(G_1,\ldots,G_n)=\vec{G}\in\prod_{j=1}^{j=n}\mathcal{G}(Z_{i_j})$, $\vec{U}\in\prod_{j=1}^{j=n}\powerset(Z_{i_j})$, $\vec{u}\in\prod_{j=1}^{j=n}Z_{i_j}$ (where $i_j\in\{1,\partial\}$).   Most of the time we are interested in some particular argument place $1\leq k\leq n$ and we write $\vec{G}[F]_k$ for the tuple $\vec{G}$ where $G_k=F$ (or $G_k$ is replaced by $F$). Similarly $\vec{u}[x]_k$ is $(u_1,\ldots,u_{k-1},x,u_{k+1},\ldots,u_n)$.

For brevity, we write $\vec{u}\preceq\vec{v}$ for the pointwise ordering statements $u_1\preceq v_1,\ldots,u_n\preceq v_n$. We also let $\vec{u}\in\vec{W}$ stand for the conjunction of componentwise membership $u_j\in W_j$, for all $j=1,\ldots,n$.

To simplify notation, we write $\Gamma\vec{u}$ for the $n$-tuple $(\Gamma u_1,\ldots,\Gamma u_n)$. For a unary map $f$ and a tuple $\vec{u}$ we write $f[\vec{u}]$ for the tuple $(f(u_1),\ldots,f(u_n))$. Note that the same notation is used for the image $f[S]=\{fx\midsp x\in S\}$ of a set under a function $f$, but context will make it clear what the intended meaning is. The convention can be nested, so that if $S$ is a set (or sequence) of tuples $\vec{u}_i$, then $f[S]$ is the set (or sequence) consisting of the elements $f[\vec{u}_i]$.

To refer to sections of relations (the sets obtained by leaving one argument place unfilled) we make use of the notation $\vec{u}[\_]_k$ which stands for the $(n-1)$-tuple $(u_1,\ldots,u_{k-1},[\_]\;,u_{k+1},\ldots,u_n)$ and similarly for tuples of sets, extending the membership convention for tuples to cases such as $\vec{u}[\_]_k\in\vec{F}[\_]_k$ and similarly for ordering relations $\vec{u}[\_]_k\preceq\vec{v}[\_]_k$. We also quantify over tuples (with, or without a hole in them), instead of resorting to an iterated quantification over the elements of the tuple, as for example in $\exists\vec{u}[\_]_k\in\vec{F}[\_]_k\exists v,w\in G\;wR\vec{u}[v]_k$.

We extend the vectorial notation to distribution types, summarily writing $\delta=(\vec{i_j};i_{n+1})$ for $(i_1,\ldots,i_n;i_{n+1})$. Then, for example, $\vec{i_j}[\partial]_k$ is the tuple with $i_k=\partial$. Furthermore, we let $\overline{i_j}=\partial$, if $i_j=1$ and $\overline{i_j}=1$, when $i_j=\partial$.
\end{rem}

\begin{lemma}\rm
\label{basic facts}
Let $\mathfrak{F}=(X,\upv,Y)$ be a  polarity and $u\in Z=X\cup Y$.
\begin{enumerate}
\item $\upv$ is increasing in each argument place (and thereby its complement $I$ is decreasing in each argument place)
\item $(\Gamma u)'=\{u\}'$ and $\Gamma u=\{u\}^{\prime\prime}$ is a Galois set
\item Galois sets are increasing, i.e. $u\in G$ implies $\Gamma u\subseteq G$
\item For a Galois set $G$, $G=\bigcup_{u\in G}\Gamma u$
\item For a Galois set $G$, $G=\bigvee_{u\in G}\Gamma u=\bigcap_{v|G}\{v\}'$ .
\item For a Galois set $G$ and any set $W$, $W^{\prime\prime}\subseteq G$ iff $W\subseteq G$.
\end{enumerate}
\end{lemma}
\begin{proof}
  By simple calculation. Proof details are included in \cite{sdl-exp}, Lemma 2.2.  For claim 4, $\bigcup_{u\in G}\Gamma u\subseteq G$ by claim 3 (Galois sets are upsets). In claim 5, given our notational conventions, the claim is that if $G\in\gpsi$, then $G=\bigcap_{G\upv y}{}\rperp\{y\}$ and if $G\in\gphi$, then $G=\bigcap_{x\upv G}\{x\}\rperp$.
\end{proof}

\begin{defn}[Closed and Open Elements]\rm
The principal upper sets of the form $\Gamma x$, with $x\in X$, will be called {\em closed}, or {\em filter} elements of $\gpsi$, while sets of the form ${}\rperp\{y\}$, with $y\in Y$, will be referred to as {\em open}, or {\em ideal} elements of $\gpsi$. Similarly for $\gphi$. A closed element $\Gamma u$ is {\em clopen} iff there exists an element $v$, with $u|v$, such that $\Gamma u=\{v\}'$.
\end{defn}
 By Lemma \ref{basic facts}, the closed elements of $\gpsi$  join-generate $\gpsi$, while the open elements meet-generate $\gpsi$ (similarly for $\gphi$).

\begin{defn}[Galois Dual Relation]\rm\label{Galois dual relations}
For a relation $R$, of sort type $\sigma$, its {\em Galois dual} relation $R'$ is the relation defined by $uR'\vec{v}$ iff $\forall w\;(wR\vec{v}\lra w|u)$. In other words, $R'\vec{v}=(R\vec{v})'$.
\end{defn}
For example, given a relation $R^{111}$ its Galois dual is the relation $R^{\partial 11}$ where for any $x,z\in X$, $R^{\partial 11}xz=(R^{111}xz)\rperp=\{y\in Y\midsp\forall u\in X\;(uR^{111}xz\lra u\upv y)\}$ and, similarly, for a relation $S^{\partial 1\partial}$ its Galois dual is the relation $S^{11\partial}$ where for any $z\in X, v\in Y$ we have $S^{11\partial}zv={}\rperp(S^{\partial 1\partial}zv)$, i.e. $xS^{11\partial}zv$ holds iff for all $y\in Y$, if $yS^{\partial 1\partial}zv$ obtains, then $x\upv y$.

\begin{defn}[Sections of Relations]\rm
\label{sections defn}
For an $(n+1)$-ary relation $R^\sigma$ and an $n$-tuple $\vec{u}$, $R^\sigma\vec{u}=\{w\midsp wR^\sigma\vec{u}\}$ is the {\em section} of $R^\sigma$ determined by $\vec{u}$. To designate a section of the relation at the $k$-th argument place we let $\vec{u}[\_]_k$ be the tuple with a hole at the $k$-th argument place. Then $wR^\sigma\vec{u}[\_]_k=\{v\midsp wR^\sigma\vec{u}[v]_k\}\subseteq Z_{i_k}$ is the $k$-th section of $R^\sigma$.
\end{defn}

We defer the definition of the category {\bf SRF}$_\tau$ of sorted residuated frames of type $\tau$ for later (see Definition \ref{tau-frames defn}), after establishing the necessary facts.

\subsection{Frame Relations and Operators}
\label{frames with relations section}
If $R^\sigma$ is a relation on a sorted residuated frame $\mathfrak{F}=(X,I,Y)$, of some sort type $\sigma=(i_{n+1};i_1\cdots i_n)$, then as in the unsorted case, $R^\sigma$ (but we shall drop the displayed sort type when clear from context)  generates a (sorted) {\em image operator} $\alpha_R$, defined by \eqref{sorted image ops}, of sort $\sigma(\alpha_R)=(i_1,\ldots,i_n;i_{n+1})$, defined by the obvious generalization of the J\'{o}nsson-Tarski image operators \cite{jt1}.
\begin{eqnarray}\label{sorted image ops}
  \alpha_R(\vec{W})&=\;\{w\in Z_{i_{n+1}}\midsp \exists \vec{w}\;(wR\vec{w}\wedge\bigwedge_{j=1}^{j=n}(w_j\in W_j))\} &=\; \bigcup_{\vec{w}\in\vec{W}}R\vec{w}
\end{eqnarray}
where for each $j$, $W_j\subseteq Z_{i_j}$ (and recall that $Z_{i_j}=X$ when $i_j=1$ and $Z_{i_j}=Y$, if $i_j=\partial$).

Thus $\alpha_R:\prod_{j=1}^{j=n}\powerset(Z_{i_j})\lra\powerset(Z_{i_{n+1}})$ is a sorted normal and completely additive function in each argument place, therefore it is residuated, i.e. for each $k$ there is a set-operator $\beta_R^k$ satisfying the condition:
\begin{equation}\label{residuation condition}
\alpha_R(\vec{W}[V]_k)\subseteq U\;\mbox{ iff }\; V\subseteq\beta_R^k(\vec{W}[U]_k)
\end{equation}
Hence $\beta_R^k(\vec{W}[U]_k)$ is the largest set $V$ s.t. $\alpha_R(\vec{W}[V]_k)\subseteq U$ and it is thereby definable by
\begin{equation}\label{def residual of alphaR}
  \beta_R^k(\vec{W}[U]_k)=\bigcup\{V\midsp \alpha_R(\vec{W}[V]_k)\subseteq U\}
\end{equation}

\begin{defn}\rm\label{overline alpha R}
$\overline{\alpha}_R$ is the closure of the restriction of $\alpha_R$ to Galois sets $\vec{F}$,
\begin{equation}\label{closure of restriction}
\overline{\alpha}_R(\vec{F})=(\alpha_R(\vec{F}))^{\prime\prime}=\left(\bigcup_{j=1,\ldots,n}^{w_j\in F_j}R\vec{w}\right)^{\prime\prime}=\bigvee_{\vec{w}\in\vec{F}}(R\vec{w})^{\prime\prime}
\end{equation}
where $F_j\in\mathcal{G}(Z_{i_j})$, for each $j\in\{1,\ldots,n\}$.
\end{defn}
In Theorem \ref{distribution from stability thm} we establish conditions under which the sorted operation $\overline{\alpha}_R$ on Galois sets is completely distributive, in each argument place.

The operator $\overline{\alpha}_R$ is  sorted  and its sorting is inherited from the sort type of $R$. For example, if $\sigma(R)=(\partial;11)$,  $\alpha_R:\powerset(X)\times\powerset(X)\lra\powerset(Y)$, hence $\overline{\alpha}_R:\gpsi\times\gpsi\lra\gphi$.
Single sorted operations
\[
\mbox{$\overline{\alpha}^1_R:\gpsi\times\gpsi\lra\gpsi$ and $\overline{\alpha}^\partial_R:\gphi\times\gphi\lra\gphi$}
\]
can be then extracted by composing appropriately with the Galois connection: $\overline{\alpha}^1_R(A,C)=(\overline{\alpha}_R(A,C))'$ (where $A,C\in\gpsi$) and, similarly, $\overline{\alpha}^\partial_R(B,D)=\overline{\alpha}_R(B',D')$ (where $B,D\in\gphi$). Similarly for the $n$-ary case and for an arbitrary distribution type.

\begin{defn}[Complex Algebra]\rm
\label{complex algebra defn}
Let $\mathfrak{F}=(X,\upv,Y,R)$ be a polarity with a relation $R$ of some sort $\sigma(R)=(i_{n+1};i_1\cdots i_n)$. The {\em complex algebra of} $\mathfrak{F}$ is the structure $\mathfrak{F}^+=(\gpsi,\overline{\alpha}^1_R)$ and its {\em dual complex algebra} is the structure $\mathfrak{F}^\partial=(\gphi,\overline{\alpha}^\partial_R)$.
\end{defn}

\noindent
Most of the time we work with the {\em dual sorted algebra of Galois sets} as it allows
\[
\left\langle(\;)\rperp:\gpsi\iso\gphi^\partial:{}\rperp(\;),\;\;\overline{\alpha}_R:\prod_{j=1}^{j=n} \mathcal{G}(Z_{i_j})\lra\mathcal{G}(Z_{i_{n+1}})\right\rangle,\hskip5mm(Z_1=X,\;Z_\partial=Y)
\]
for considering sorted operations that distribute over joins in each argument place (which are either joins of $\gpsi$, or of $\gphi$, depending on the sort type of the operation). Single-sorted normal operators are then  extracted in the complex algebra by composition with the Galois maps, as indicated above.

\begin{defn}[Conjugates]\rm
\label{conjugates defn}
Let $\alpha$ be an image operator (generated by some relation $R$) of sort type $\sigma(\alpha)=(\vec{i_j};i_{n+1})$ and $\overline{\alpha}$ the closure of its restriction to Galois sets in each argument place, as defined above.
A function $\overline{\gamma}^k$  on Galois sets, of sort type $\sigma(\overline{\gamma}^k)=(\vec{i_j}[\overline{i_{n+1}}]_k;\overline{i_k})= (i_1,\ldots,i_{k-1},\overline{i_{n+1}},i_{k+1},\ldots,i_n;\overline{i_k})$
(where $\overline{i_j}=\partial$ if $i_j=1$ and $\overline{i_j}=1$ when $i_j=\partial$) is a {\em conjugate} of $\overline{\alpha}$ at the $k$-th argument place (or a $k$-conjugate) iff the following condition holds
\begin{equation}\label{def conjugates}
\hspace*{-2mm}\overline{\alpha}(\vec{F})\subseteq G\;\mbox{ iff }
\overline{\gamma}^k(\vec{F}[G']_k)\subseteq F_k'
\end{equation}
for all Galois sets $F_j\in\mathcal{G}(Z_{i_j})$ and $G\in\mathcal{G}(Z_{i_{n+1}})$.
\end{defn}

It follows from the definition of a conjugate function that $\overline{\gamma}$ is a $k$-conjugate of $\overline{\alpha}$ iff $\overline{\alpha}$ is one of $\overline{\gamma}$ and we thus call $\overline{\alpha},\overline{\gamma}$ $k$-conjugates. Note that the priming notation for both maps of the duality $(\;)\rperp:\gpsi\iso\gphi^\partial:{}\rperp(\;)$ packs together, in one form, four distinct (due to sorting) cases of conjugacy.

\begin{ex}\rm
\label{example and strategy ex 1}
In the case of a ternary relation $R^{111}$ of the indicated sort type, an image operator $\alpha_R=\bigodot:\powerset(X)\times\powerset(X)\lra\powerset(X)$ is generated.  Designate the closure of its restriction to Galois stable sets by $\bigovert:\gpsi\times\gpsi\lra\gpsi$.
Then $\overline{\alpha}=\bigovert$ is of sort type $\sigma(\bigovert)=(1,1;1)$. If $\overline{\gamma}^2_R=\triangleright:\gpsi\times\gphi\lra\gphi$, with $\sigma(\triangleright)=(1,\partial;\partial)$,  then $\bigovert$, $\triangleright$ are {\em conjugates} iff for any Galois stable sets $A,F,C\in\gpsi$ it holds that $A\bigovert F\subseteq C$ iff $A\triangleright C'\subseteq F'$.

Note that, given an operator $\triangleright:\gpsi\times\gphi\lra\gphi$, if we now define $\Ra\;:\gpsi\times\gpsi\lra\gpsi$ by $A\Ra C=(A\triangleright C')'={}\rperp(A\triangleright C\rperp)$, it is immediate that $\bigovert, \triangleright$ are conjugates iff $\bigovert,\Ra$ are residuated. In other words
\[
A\mbox{$\bigovert$}F\subseteq C\;\mbox{ iff }\;A\triangleright C'\subseteq F'\;\mbox{ iff }\;F\subseteq A\Ra C
\]
\end{ex}

\begin{lemma}\rm\label{conjugate-residual}
The following are equivalent.
\begin{enumerate}
\item[1)] $\overline{\alpha}_R$ distributes over any joins of Galois sets at the $k$-th argument place
\item[2)] $\overline{\alpha}_R$ has a $k$-conjugate $\overline{\gamma}_R^k$ defined on Galois sets by
  \[
  \overline{\gamma}_R^k(\vec{F})=\bigcap\{G\midsp \overline{\alpha}_R(\vec{F}[G']_k)\subseteq F'_k\}
  \]
\item[3)] $\overline{\alpha}_R$ has a $k$-residual $\overline{\beta}_R^k$  defined on Galois sets by
  \[
      \overline{\beta}_R^k(\vec{F}[G]_k)= (\overline{\gamma}_R^k(\vec{F}[G']_k))'=\bigvee\{G^\prime\midsp \overline{\alpha}_R(\vec{F}[G']_k)\subseteq F'_k\}
  \]
\end{enumerate}
\end{lemma}
\begin{proof}
Existence of a $k$-residual is equivalent to distribution over arbitrary joins and the residual is defined by
\[
\overline{\beta}^k_R(\ldots,F_{k-1},H,F_{k+1},\ldots)=\bigvee\{G\midsp\overline{\alpha}_R(\ldots,F_{k-1},G,F_{k+1},\ldots)\subseteq H\}
\]
We show that the distributivity assumption 1) implies that 2) and 3) are equivalent, i.e. that
\[
\overline{\alpha}_R(\vec{F}[G]_k)\subseteq H\;\mbox{ iff }\;
\overline{\gamma}^k_R(\vec{F}[H']_k)\subseteq G'\;\mbox{ iff }\;G\subseteq\overline{\beta}^k_R(\vec{F}[H]_k)
\]

We illustrate the proof for the unary case only, as the other parameters remain idle in the argument.

Assume $\overline{\alpha}_R(G)\subseteq H$ and let $\overline{\gamma}_R(H^\prime)=\bigcap\{E\midsp\overline{\alpha}_R(E')\subseteq H\}$, a Galois set by definition, given that $G,H,E$ are assumed to be Galois sets. Then $G^\prime$ is in the set whose intersection is taken. Hence $\overline{\gamma}_R(H^\prime)\subseteq G^\prime$ follows from the definition of $\overline{\gamma}_R$. It also follows by  definition that $G\subseteq\overline{\beta}_R(H)=(\overline{\gamma}_R(H^\prime))^\prime$.

Assuming $G\subseteq\overline{\beta}_R(H)$ we obtain by definition that $G\subseteq (\overline{\gamma}_R(H^\prime))^\prime$, hence  $G\subseteq\bigvee\{E^\prime\midsp\overline{\alpha}_R(E^\prime)\subseteq H\}$, using the definition of $\overline{\gamma}_R$ and duality. Hence by the distributivity assumption $\overline{\alpha}_R(G)\subseteq\bigvee\{\overline{\alpha}_R(E^\prime)\midsp \overline{\alpha}_R(E^\prime)\subseteq H\}\subseteq H$. This establishes that
$\overline{\alpha}_R(G)\subseteq H$ iff $\overline{\gamma}_R(H^\prime)\subseteq G^\prime$ iff $G\subseteq\overline{\beta}_R(H)$, qed.
\end{proof}

\begin{thm}\rm
\label{distribution from stability thm}
Let $\mathfrak{F}=(X,\upv,Y,R)$ be a frame with an $(n+1)$-ary sorted relation, of some sort  $\sigma(R)=(i_{n+1};\vec{i_j})$ and assume that for any $w\in Z_{\overline{i_{n+1}}}$ and any $(n-1)$-tuple $\vec{p}[\_]_k$ with $p_j\in Z_{i_j}$, for each $j\in\{1,\ldots,n\}\setminus\{k\}$, the sections $wR'\vec{p}[\_]_k$ of the Galois dual relation $R'$ of $R$ are Galois sets. Then $\overline{\alpha}_R$ distributes at the $k$-th argument place over arbitrary joins in $\mathcal{G}(Z_{i_k})$.
\end{thm}
\begin{proof}
  Define the relation $T$ from $R$ by setting,
  \[
  vT\vec{p}[w]_k\;\mbox{ iff }\; w\in (vR'\vec{p}[\_]_k)'
  \]
  Then use equation \eqref{k-conjugate relations constraint} below, to define a relation $S$
   \begin{equation}\label{k-conjugate relations constraint}
   \forall v\in Z_{\overline{i_{n+1}}}\forall \vec{p}[\_]_k\in\vec{Z_{i_j}}[\_]_k\forall w\in Z_{\overline{i_k}}\left( vT\vec{p}[w]_k\;\leftrightarrow\;wS\vec{p}[v]_k \right)
 \end{equation}
 Note that the sort type of $S$, as defined, is $\sigma(S)=(\overline{i_k};\vec{i_j}[\overline{i_{n+1}}]_k)$ . Let $\overline{\eta}_S$ be the closure of the restriction of the image operator $\eta_S$ to Galois sets, according to the sort type of $S$. We show that $\overline{\alpha}_R$ and $\overline{\eta}_S$ are $k$-conjugates. To establish the conjugacy condition $\overline{\alpha}_R(\vec{F})\subseteq G$ iff $\overline{\eta}_S(\vec{F}[G']_k)\subseteq F'_k$ it suffices by Lemma \ref{basic facts} to verify that $\alpha_R(\vec{F})\subseteq G$ iff $\eta_S(\vec{F}[G']_k)\subseteq F'_k$.
 We have
\begin{tabbing}
$\alpha_R(\vec{F})\subseteq G$\hskip5mm\= iff\hskip5mm\= $\bigcup_{\vec{p}\in\vec{F}}R\vec{p}\;\subseteq G$ \hskip10mm\= iff\hskip2mm\=$\forall\vec{p}\;(\vec{p}\in\vec{F}\;\lra\; (R\vec{p}\subseteq G))$\\
\hskip5mm\=iff\hskip2mm\= $\forall\vec{p}\;(\vec{p}\in\vec{F}\;\lra\;(G'\subseteq R'\vec{p}))$\\
 \>iff\> $\forall\vec{p}\;(\vec{p}\in\vec{F}\;\lra\;\forall v\in Z_{\overline{i_{n+1}}}(G|v\lra vR'\vec{p}))$\\
\>iff\> $\forall\vec{p}\forall v\in Z_{\overline{i_{n+1}}}\;(\vec{p}[\_]_k\in\vec{F}[\_]_k\wedge p_k\in F_k\wedge G|v\;\lra\; vR'\vec{p}[p_k]_k)$\\
\>iff\>  $\forall\vec{p}\forall v\in Z_{\overline{i_{n+1}}}\;(\vec{p}[\_]_k\in\vec{F}[\_]_k \wedge G|v\;\lra\; (p_k\in F_k\lra vR'\vec{p}[p_k]_k))$\\
\>iff\> $\forall\vec{p}[\_]_k\forall v\in Z_{\overline{i_{n+1}}}\;(\vec{p}[\_]_k\in\vec{F}[\_]_k \wedge G|v\;\lra\; (F_k\subseteq vR'\vec{p}[\_]_k))$\\
\>\> (using the hypothesis that the $k$-th sections of $R'$ are Galois sets)\\
\>iff\> $\forall\vec{p}[\_]_k\forall v\in Z_{\overline{i_{n+1}}}\;(\vec{p}[\_]_k\in\vec{F}[\_]_k \wedge G|v\;\lra\; (\;(vR'\vec{p}[\_]_k)'\subseteq F^\prime_k))$\\
\>\> (using the definition of $T$)\\
\>iff\> $\forall\vec{p}[\_]_k\forall v\in Z_{\overline{i_{n+1}}}\;\left(\vec{p}[\_]_k\in\vec{F}[\_]_k \wedge G|v\;\lra\; \forall w\in Z_{\overline{i_k}} (\;vT\vec{p}[w]_k\lra F_k|w)\right)$\\
\>iff\> $\forall\vec{p}[\_]_k\forall v\in Z_{\overline{i_{n+1}}}\forall w\in Z_{\overline{i_k}}\; \left(vT\vec{p}[w]_k \wedge \vec{p}[\_]_k\in\vec{F}[\_]_k \wedge G|v\;\lra\; F_k|w\right)$\\
\>\> (using the definition of $S$)\\
\>iff\> $\forall\vec{p}[\_]_k\forall v\in Z_{\overline{i_{n+1}}}\forall w\in Z_{\overline{i_k}}\; \left(wS\vec{p}[v]_k \wedge \vec{p}[\_]_k\in\vec{F}[\_]_k \wedge G|v\;\lra\; F_k|w\right)$\\
\>iff\> $\bigcup_{\vec{p}[v]_k\in\vec{F}[G']_k} S\vec{p}[v]_k\;\subseteq F^\prime_k$\\
\>iff\> $\eta_S(\vec{F}[G']_k)\subseteq F^\prime_k$
\end{tabbing}
Hence $\overline{\alpha}_R$ and $\overline{\eta}_S$ are $k$-conjugates.
 Consequently, by Lemma \ref{conjugate-residual}, $\overline{\alpha}_R$ distributes at the $k$-th argument place over arbitrary joins in $\mathcal{G}(Z_{i_k})$.
\end{proof}

\begin{defn}\rm
\label{betakR defn}
Let $\beta^k_{R/}$ be the restriction of $\beta^k_R$ of equation \eqref{def residual of alphaR} to Galois sets, according to its sort type, explicitly defined by \eqref{betakR}:

\begin{equation}\label{betakR}
  \beta^k_{R/}(\vec{E}[G]_k)=\bigcup\{F\in\mathcal{G}(Z_{i_k})\midsp\alpha_R(\vec{E}[F]_k)\subseteq G\}
\end{equation}
\end{defn}

\begin{thm}
\label{beta1R residual}
  \rm
  If $\overline{\alpha}_R$ is residuated in the $k$-th argument place, then $\beta^k_{R/}$ is its residual and $ \beta^k_{R/}(\vec{E}[G]_k)$ is a Galois set, i.e. the union in equation \eqref{betakR} is actually a join in $\mathcal{G}(Z_{i_k})$.
\end{thm}
\begin{proof}
  We illustrate the proof for the unary case only, since the other parameters that may exist remain idle in the argument. In the unary case, $\beta_{R/}(G)=\bigcup\{F\midsp\alpha_R(F)\subseteq G\}$, for Galois sets $F,G$.

  Note first that $\overline{\alpha}_R(F)\subseteq G$ iff $F\subseteq\beta_{R/}(G)$. Left-to-right is obvious by definition and by the fact that for a Galois set $G$ and any set $U$, $U^{\prime\prime}\subseteq G$ iff $U\subseteq G$.  If $F\subseteq\beta_{R/}(G)\subseteq\beta_R(G)$, then by residuation $\alpha_R(F)\subseteq G$. Given that $G$ is a Galois set, it follows $\overline{\alpha}_R(F)\subseteq G$.

 If indeed $\overline{\alpha}_R$ is residuated on Galois sets with a map $\overline{\beta}_R$, then the residual is defined by $\overline{\beta}_R(G)=\bigvee\{F\midsp \overline{\alpha}_R(F)\subseteq G\}=\bigvee\{F\midsp \alpha_R(F)\subseteq G\}$ and this is precisely the closure of $\beta_{R/}(G)=\bigcup\{F\midsp\alpha_R(F)\subseteq G\}$. But in that case we obtain $F\subseteq\overline{\beta}_R(G)$ iff $\overline{\alpha}_R(F)\subseteq G$ iff $\alpha_R(F)\subseteq G$ iff $F\subseteq\beta_{R/}(G)$ and setting $F=\overline{\beta}_R(G)$ it follows that $\overline{\beta}_R(G)\subseteq\beta_{R/}(G)\subseteq\overline{\beta}_R(G)$.
\end{proof}

\begin{lemma}\rm
\label{equiv defn of betakr}
$\beta^k_{R/}$ is equivalently defined by \eqref{beta equiv 1} and by \eqref{beta equiv 2}
\begin{eqnarray}
\beta^k_{R/}(\vec{E}[G]_k)&=&\bigcup\{\Gamma u\in\mathcal{G}(Z_{i_k})\midsp\alpha_R(\vec{E}[\Gamma u]_k)\subseteq G\}
\label{beta equiv 1}\\
\beta^k_{R/}(\vec{E}[G]_k)&=&\{ u\in Z_{i_k}\midsp\alpha_R(\vec{E}[\Gamma u]_k)\subseteq G\}
\label{beta equiv 2}
\end{eqnarray}
\end{lemma}
\begin{proof}
  $\beta^k_{R/}$ is defined by equation \eqref{betakR}, so if $u\in \beta^k_{R/}(\vec{E}[G]_k)$, let $F\in \mathcal{G}(Z_{i_k})$ be such that $u\in F$ and $\alpha_R(\vec{E}[F]_k)\subseteq G$. Then $\Gamma u\subseteq F$ and by monotonicity of $\alpha_R$ we have
  $
  \alpha_R(\vec{E}[\Gamma u]_k)\subseteq\alpha_R(\vec{E}[F]_k)\subseteq G
   $
   and this establishes the left-to-right inclusion for the first identity of the lemma. The converse inclusion is obvious since $\Gamma u$ is a Galois set.

  For the second identity, the inclusion right-to-left is obvious. Now if $u$ is such that $\alpha_R(\vec{E}[\Gamma u]_k)\subseteq G$ and $u\preceq w$, then $\Gamma w\subseteq\Gamma u$ and then by monotonicity of $\alpha_R$ it follows that $\alpha_R(\vec{E}[\Gamma w]_k)\subseteq \alpha_R(\vec{E}[\Gamma u]_k)\subseteq G$.

  This shows that  $\bigcup\{\Gamma u\in\mathcal{G}(Z_{i_k})\midsp\alpha_R(\vec{E}[\Gamma u]_k)\subseteq G\}$ is contained in the set $\{ u\in Z_{i_k}\midsp\alpha_R(\vec{E}[\Gamma u]_k)\subseteq G\}$, and given the first part of the lemma, the second identity obtains as well.
\end{proof}

We summarize our so far results with the following observations.

Let $\mathbb{C}_\tau$ be the class of sorted residuated frames (equivalently, polarities) with relations $R_\sigma$ of sort type $\sigma$, for each $\sigma=(i_{n+1};\vec{i_j})\in\{1,\partial\}^{n+1}$ in the similarity type $\tau$. Assume the stability axiom below for $\mathbb{C}_\tau$.
\begin{itemize}
\item For each relation $R$ of type $\sigma=(i_{n+1};\vec{i_j})$ and each $w\in Z_{\overline{i_{n+1}}}$ and $\vec{u}[\_]_k$ with $u_j\in Z_{i_j}$ for each $j\in\{1,\ldots,n\}\setminus\{k\}$, the  $k$-th section $wR^\prime\vec{u}[\_]_k$ of the Galois dual relation $R'$ of $R$ is a Galois set, for each $k=1,\ldots,n$.
\end{itemize}
Let $\alpha_R$ be the classical sorted image operator generated by $R$, as in equation \eqref{sorted image ops}, and $\beta^k_R$ its $k$-residual for any $k=1,\ldots,n$, defined as usual by equation \eqref{def residual of alphaR}. Then
\begin{enumerate}
\item the closure $\overline{\alpha}_R$ (Definition \ref{overline alpha R}) of the restriction of $\alpha_R$ to Galois sets is residuated at the $k$-th argument place with the restriction $\beta^k_{R/}$ (Definition \ref{betakR defn}) of $\beta^k_R$ to Galois sets (Lemma \ref{conjugate-residual}, Theorem \ref{distribution from stability thm}, Theorem \ref{beta1R residual})
\item a completely normal operator $\overline{\alpha}^1_R:\gpsi^n\lra\gpsi$ of distribution type $\delta=(\vec{i_j};i_{n+1})$ is obtained by composition with the Galois connection
    \[
    \overline{\alpha}^1_R(A_1,\ldots,A_n)=\left\{
    \begin{array}{cl}
    \overline{\alpha}_R(\ldots,\underbrace{A_j}_{i_j=1},\ldots,\underbrace{A'_r}_{i_r=\partial},\ldots) &\mbox{ if }i_{n+1}=1\\
    \left(\overline{\alpha}_R(\ldots,\underbrace{A_j}_{i_j=1},\ldots,\underbrace{A'_r}_{i_r=\partial},\ldots)\right)^\prime &\mbox{ if }i_{n+1}=\partial
    \end{array}
    \right.
    \]
\item similarly for its dual operator $\overline{\alpha}^\partial_R:\gphi^n\lra\gphi$.
\end{enumerate}

We list in Table \ref{frame axioms table} the frame axioms we shall assume in the sequel, for a sorted residuated frame with relations $\mathfrak{F}=(X,I, Y, (R_\sigma)_{\sigma\in\tau})$.
\begin{table}[t]
\label{frame axioms table}
\caption{Axioms for Sorted Residuated Frames of similarity type $\tau$}
\begin{enumerate}
  \item[FAx1)] The frame is separated
  \item[FAx2)] For each $\sigma=(\vec{i_j};i_{n+1})$ in the similarity type $\tau$, each $\vec{u}\in\prod_{j=1}^{j=n}Z_{i_j}$, $R_\sigma\vec{u}$ is a closed element of $\mathcal{G}(Z_{i_{n+1}})$
  \item[FAx3)] For each $\sigma=(\vec{i_j};i_{n+1})$ in the similarity type $\tau$, each $w\in Z_{i_{n+1}}$, the $n$-ary relation $wR_\sigma$ is decreasing in every argument place
  \item[FAx4)] All sections of the Galois dual relations $R'_\sigma$ of $R_\sigma$, for each $\sigma$ in $\tau$, are Galois sets
\end{enumerate}
\end{table}

Note that axioms 1 and 2 imply that there is a (sorted) function $\widehat{f}_R$ on the points of the frame such that $\widehat{f}_R(\vec{u})=w$ iff $R\vec{u}=\Gamma w$. The following immediate observation will be useful in the sequel.

\begin{lemma}\rm\label{alpha bar on closed}
Let $\mathfrak{F}$ be a frame of similarity type $\tau$ and assume that axioms 1-3 in Table \ref{frame axioms table} hold. Then for a frame relation $R$ of type $\sigma$ in $\tau$,  $\overline{\alpha}_R(\Gamma\vec{u})=R\vec{u}=\alpha_R(\Gamma\vec{u})=\Gamma(\widehat{f}_R(\vec{u}))$.
\end{lemma}
\begin{proof}
  By definition \eqref{sorted image ops}, $\alpha_R(\Gamma\vec{u})=\bigcup_{\vec{u}\leq\vec{w}}R\vec{w}$. By axiom 3, $\bigcup_{\vec{u}\leq\vec{w}}R\vec{w}=R\vec{u}$, which is a closed element by axiom 2, generated by a unique point $w=\widehat{f}_R(\vec{u})$, by axiom 1, so that $\alpha_R(\Gamma\vec{u})=R\vec{u}=\Gamma w=(\alpha_R(\Gamma\vec{u}))^{\prime\prime}=\overline{\alpha}_R(\Gamma\vec{u})$, where $\Gamma w=\Gamma(\widehat{f}_R(\vec{u}))$.
\end{proof}
The axiomatization will be strengthened in Section \ref{stone section}, in order to be able to carry out a Stone duality proof.

\subsection{Weak Bounded Morphisms}
\label{bounded section}
Recall that  a {\em bounded morphism} $p:(W_1,R_1)\lra(W_2,R_2)$, for classical Kripke frames, is defined as a map that preserves the frame relation, i.e. $uR_1v$ implies that $p(u)R_2p(v)$ and so that its inverse $p^{-1}$ is a homomorphism of the dual modal algebras $p^{-1}:(\powerset(W_2),\largediamond_{\!2})\lra (\powerset(W_1),\largediamond_{\!1})$, i.e. such that $p^{-1}(\largediamond_{\!2}V)=\largediamond_{\!1}(p^{-1}V)$. This can be re-written as the familiar first-order condition typically used to define bounded morphisms.

For sorted frames, when their dual sorted residuated modal algebras are of interest, morphisms of sorted frames can be taken to be the natural generalization of bounded morphisms to the sorted case, to wit a pair of maps $(p,q):(X_2,I_2,Y_2)\lra(X_1,I_1,Y_1)$,  such that their inverses commute with the residuated set-operators $\largediamond_{\!1},\lbbox_{\!1}$ and $\largediamond_{\!2},\lbbox_{\!2}$ (equivalently, with the Galois connections). One direction of the required inclusions is ensured by requiring preservation of the frame relation, as in the unsorted case. Since inverse maps preserve unions and every set can be written as the union of the singletons of its elements, the reverse inclusion $\largediamond_{\!2}p^{-1}(U)\supseteq q^{-1}(\largediamond_{\!1} U)$ will hold iff it holds for singletons $\largediamond_{\!2}p^{-1}(\{x\})\supseteq q^{-1}(\largediamond_{\!1} \{x\})$. Rephrasing and expressing it as a first-order condition we obtain  condition \eqref{bounded-diamond}. Similarly for the other reverse inclusion, after replacing boxes with diamonds, working with co-atoms $-\{x\}$ and contraposing a number of times we obtain the equivalent first-order condition \eqref{bounded-box}.
\begin{eqnarray}
\forall x\in X_1\forall y'\in Y_2(xI_1 q(y')\lra\exists x'\in X_2(x = p(x')\wedge x'I_2y'))\label{bounded-diamond}\\
\forall x'\in X_2\forall y\in Y_1(p(x')I_1y\lra\exists y'\in Y_2(y= q(y')\wedge x'I_2y')\label{bounded-box}
\end{eqnarray}
We then arrive at the natural generalization and a {\em sorted bounded morphism} $(p,q):(X_2,I_2,Y_2)\lra(X_1,I_1,Y_1)$ is defined as a pair of maps $p:X_2\lra X_1$, $q:Y_2\lra Y_1$ such that the relation preservation condition \eqref{preservation} below,
\begin{equation}
\label{preservation}
\forall x'\in X_2\forall y'\in Y_2\;(x'I_2y'\lra p(x')I_1q(y'))
\end{equation}
as well as conditions \eqref{bounded-diamond} and \eqref{bounded-box} hold. Note that sorted bounded morphisms preserve the closure operators
\[
  p^{-1}(\lbbox_1\largediamond_1 U)=\lbbox_2\largediamond_2 p^{-1}(U)\;\mbox{ and }\; q^{-1}(\largesquare_1\lbdiamond_1 V)=\largesquare_2\lbdiamond_2 q^{-1}(V)
\]
therefore they preserve arbitrary joins, since these are closures of unions and as inverse maps preserve both unions and intersections, sorted bounded morphisms are homomorphisms of the complete lattices of Galois stable and co-stable sets.

\subsubsection{Morphisms for Sorted Residuated Frames}
\label{polarity morphisms section}
Singletons are atoms of the powerset Boolean algebras and they join-generate any subset, i.e. $U=\bigcup_{u\in U}\{u\}$, and this was used in computing the first-order conditions \eqref{bounded-diamond}, \eqref{bounded-box} for sorted bounded morphisms. For stable and co-stable sets, join generators are the closed elements $\Gamma x\;(x\in X)$ and $\Gamma y\; (y\in Y)$ so that we have, respectively, $A=\bigvee_{x\in A}\Gamma x=\bigcup_{x\in A}\Gamma x$, using Lemma \ref{basic facts}. We have, for any $x\in X$,

  \begin{tabbing}
  $q^{-1}(\largediamond_{\!1}\Gamma x)\subseteq\largediamond_{\!2} p^{-1}(\Gamma x)$\\
  \hspace*{1cm}\=iff\hskip1mm\=
 $\forall y'\in Y_2(y'\in q^{-1}(\largediamond_{\!1}\Gamma x)\;\Ra\;y'\in\largediamond_{\!2} p^{-1}(\Gamma x))$\\
  \>iff\> $\forall y'\in Y_2(q(y')\in\largediamond_{\!1}\Gamma x\;\Ra\;y'\in\largediamond_{\!2} p^{-1}(\Gamma x))$\\
  \>iff\> $\forall y'\in Y_2(\exists z\in X_1(zI_1q(y')\wedge x\leq z)\;\Ra\; y'\in\largediamond_{\!2} p^{-1}(\Gamma x))$\\
  \>iff\> $\forall y'\in Y_2(xI_1q(y')\;\Ra\; y'\in\largediamond_{\!2} p^{-1}(\Gamma x))$\\
  \>iff\> $\forall y'\in Y_2(xI_1q(y')\;\Ra\;\exists x'\in X_2(x\leq p(x')\wedge x' I_2 y')$
  \end{tabbing}
and this is the weakened version of \eqref{bounded-diamond} we shall need. We point out that
the proof used the fact that the frame relation $I$ is decreasing in both argument places (Lemma \ref{basic facts}), hence $\exists z\in X_1(zI_1q(y')\wedge x\leq z)$ iff $xI_1q(y')$.
\begin{defn}\rm
\label{pi and pi inverse defn}
If $(p,q):(X_2,I_2,Y_2)\ra(X_1,I_1,Y_1)$, with $p:X_2\ra X_1$ and $q:Y_2\ra Y_1$,
then we let $\pi=(p,q)$ and we define $\pi^{-1}$
by setting
\[
\pi^{-1}(W)=\left\{\begin{array}{cl}p^{-1}(W)\in\powerset(X_2) &\mbox{ if }W\subseteq X_1\\
q^{-1}(W)\in\powerset(Y_2) & \mbox{ if }W\subseteq Y_1\end{array}\right.
\]
 Similarly, we let
\[
\pi(w)=\left\{\begin{array}{cl}p(w)\in X_1 &\mbox{ if }w\in X_2\\ q(w)\in Y_1 &\mbox{ if }w\in Y_2
\end{array}\right.
\]
\end{defn}

\begin{lemma}\rm
\label{cond2-lemma}
If $\pi=(p,q):(X_2,I_2,Y_2)\lra(X_1,I_1,Y_1)$ is a pair of maps \mbox{$p:X_2\lra X_1$,} $q:Y_2\lra Y_1$,
then the following are equivalent
\begin{enumerate}
  \item For any increasing subset $A\subseteq X_1$, $\pi^{-1}(\largediamond_1 A)\subseteq\largediamond_2 \pi^{-1}(A)$
  \item For any $x\in X_1$, $\pi^{-1}(\largediamond_1\Gamma x)\subseteq\largediamond_2 \pi^{-1}(\Gamma x)$
  \item $\forall x\in X_1\forall y'\in Y_2(xI_1 \pi(y')\lra\exists x'\in X_2(x\leq \pi(x')\wedge x'I_2y')  )$
\end{enumerate}
\end{lemma}
\begin{proof}
\begin{description}
  \item[(1)$\Ra$(2)] Immediate, since $\Gamma x=\{z\in X_1\midsp x\leq z\}\subseteq X_1$ is increasing.
  \item[(2)$\Leftrightarrow$(3)] This was shown above.
  \item[(3)$\Ra$(1)] Let $y'\in q^{-1}(\largediamond_1 A)$, i.e. $q(y')\in\largediamond_1 A$ and let then $x\in X_1$ be such that $xI_1q(y')$ and $x\in A$. From $xI_1q(y')$ and condition 3) we obtain that there exists $x'\in X_2$ such that $x'I_2y'$ and $x\leq p(x')$. Given the assumption that $A$ is an increasing subset and since $x\in A$ it follows that $p(x')\in A$, as well. This shows that $y'\in\largediamond_2p^{-1}(A)$.
\end{description}
Therefore, (1)$\Leftrightarrow$(2)$\Leftrightarrow$(3).
\end{proof}

Similarly, we obtain the following lemma.

\begin{lemma}
\label{cond3-lemma}
  \rm
 If $\pi=(p,q):(X_2,I_2,Y_2)\lra(X_1,I_1,Y_1)$ is a pair of maps \mbox{$p:X_2\lra X_1$,} $q:Y_2\lra Y_1$,
then the following are equivalent
\begin{enumerate}
  \item For any decreasing subset $B\subseteq Y_1$, $\lbbox_{\!2} \pi^{-1}(B)\subseteq \pi^{-1}(\lbbox_{\!1} B)$
  \item For any $y\in Y_1$, $\lbbox_{\!2} \pi^{-1}(-\Gamma y)\subseteq \pi^{-1}(\lbbox_{\!1} (-\Gamma y))$
  \item $\forall x'\in X_2\forall y\in Y_1(\pi(x')I_1y\lra\exists y'\in Y_2(y\leq \pi(y')\wedge x'I_2y')$\hfill$\largesquare$
\end{enumerate}
\end{lemma}
Note that case 3) of Lemma \ref{cond3-lemma} is a weakened analogue of \eqref{bounded-box}.

\begin{defn}
  \rm\label{sorted weak bounded morphisms}
If $\pi=(p,q):(X_2,I_2,Y_2)\lra(X_1,I_1,Y_1)$ is a pair of maps \mbox{$p:X_2\lra X_1$,} $q:Y_2\lra Y_1$,
then  $(p,q)$ will be called a {\em (sorted) weak bounded morphism} iff
\begin{enumerate}
  \item $\forall x'\in X_2\forall y'\in Y_2\;(x'I_2y'\lra p(x')I_1q(y'))$
  \item $\forall x\in X_1\forall y'\in Y_2(xI_1 q(y')\lra\exists x'\in X_2(x\leq p(x')\wedge x'I_2y')  )$
  \item $\forall x'\in X_2\forall y\in Y_1(p(x')I_1y\lra\exists y'\in Y_2(y\leq q(y')\wedge x'I_2y')$
\end{enumerate}
\end{defn}

\begin{coro}\rm
\label{weak bounded are complete lattice homos coro}
The inverse $\pi^{-1}=(p,q)^{-1}$ of a weak bounded morphism is a complete lattice homomorphism of the lattices of Galois stable sets of sorted residuated frames.
\end{coro}
\begin{proof}
 The proof is an immediate consequence of Lemmas \ref{cond2-lemma} and \ref{cond3-lemma}.
\end{proof}

Summarizing, we have shown that all squares in the diagrams in the middle and right below commute
  \[
  \xymatrix{
(X_1,\;\upv_1,\;Y_1)       & \powerset^\uparrow(X_1)\ar@<0.5ex>[r]^{\Diamond_1}\ar[d]^{p^{-1}}    &\powerset^\downarrow(Y_1)\ar[d]^{q^{-1}} \ar@<0.5ex>[l]^{\bbox_1}
                     & \powerset^\uparrow(X_1)\ar@<0.5ex>[r]^{(\;)^{\upv_1}}\ar[d]^{p^{-1}}  &\powerset^\uparrow(Y_1)\ar[d]^{q^{-1}} \ar@<0.5ex>[l]^{{}^{\upv_1}(\;)}
\\
(X_2,\upv_2,Y_2)\ar@<3ex>[u]^{p}\ar@<-3ex>[u]_{q}
                     & \powerset^\uparrow(X_2)\ar@<0.5ex>[r]^{\Diamond_2}      &\powerset^\downarrow(Y_2)\ar@<0.5ex>[l]^{\bbox_2}
                     & \powerset^\uparrow(X_2)\ar@<0.5ex>[r]^{(\;)^{\upv_2}}  &\powerset^\uparrow(Y_2)\ar@<0.5ex>[l]^{{}^{\upv_2}}
  }
  \]
where $\powerset^\uparrow$, $\powerset^\downarrow$ deliver the set of increasing and decreasing, respectively, subsets.
\begin{rem}\rm
Goldblatt \cite{goldblatt-morphisms2019} was first to propose bounded morphisms for polarities and our definition is equivalent to his. For morphisms of frames with relations (restricted in \cite{goldblatt-morphisms2019} to relations of sort type that generates either join-preserving or meet-preserving operators only) we will diverge from his definition, as the relations Goldblatt considers on frames, expanding on Gehrke's \cite{mai-gen}, do not coincide with ours and they can be in fact construed as the Galois duals of the frame relations we consider. In \cite{kata2z} we discussed the connections between our approach and Gehrke's generalized Kripke frames approach.
\end{rem}

\subsubsection{Morphisms for Frames with Relations}
\label{morphisms for relations section}
Let $\pi=(p,q):(X_2,I_2,Y_2,(S_\sigma)_{\sigma\in \tau})\lra(X_1,I_1,Y_1,(R_\sigma)_{\sigma\in \tau})$ be a weak bounded morphism of frames of the same similarity type $\tau$ and
  let $R_\sigma,S_\sigma$ be corresponding relations in the two frames, of the same sort type. For simplicity, we omit the subscript $\sigma$ in the sequel.

\begin{prop}\rm
\label{condition for frame morphisms prop}
  If for any $\vec{u}$ it holds that $\pi^{-1}\overline{\alpha}_R(\Gamma\vec{u})=\overline{\alpha}_S(\pi^{-1}[\Gamma\vec{u}])$, then for any tuple $\vec{F}$ of Galois sets of the required sort $\pi^{-1}\overline{\alpha}_R(\vec{F})=\overline{\alpha}_S(\pi^{-1}[\vec{F}])$.
\end{prop}
\begin{proof}
We have
\begin{tabbing}
$\pi^{-1}\overline{\alpha}_R(\vec{F})$\hskip2mm\==\hskip1mm\= $\pi^{-1}\bigvee_{\vec{u}\in\vec{F}}\overline{\alpha}_R(\Gamma\vec{u})$\hskip1.4cm\= By Lemma \ref{basic facts} and Theorem \ref{distribution from stability thm}\\
\>=\> $\bigvee_{\vec{u}\in\vec{F}}\pi^{-1}\overline{\alpha}_R(\Gamma\vec{u})$\> By Corollary \ref{weak bounded are complete lattice homos coro}\\
\>=\> $\bigvee_{\vec{u}\in\vec{F}}\overline{\alpha}_S(\pi^{-1}[\Gamma\vec{u}])$\> By the hypothesis of the Lemma\\
\>=\> $\bigvee_{\vec{u}\in\vec{F}}\bigvee_{\vec{w}\in\pi^{-1}[\Gamma \vec{u}]}\overline{\alpha}_S(\Gamma\vec{w})$ \> By Lemma \ref{basic facts} and Theorem \ref{distribution from stability thm}\\
\>=\> $\bigvee_{\vec{u}\leq\pi[\vec{w}]}^{\vec{u}\in\vec{F}}\overline{\alpha}_S(\Gamma\vec{w})$ \\
\>=\> $\bigvee_{\vec{w}\in\pi^{-1}[\vec{F}]}\overline{\alpha}_S(\Gamma\vec{w})$\>=\ $\overline{\alpha}_S(\pi^{-1}[\vec{F}])$
\end{tabbing}
For the last line, note that if $\vec{u}\in\vec{F}$ and $\vec{u}\leq\pi[\vec{w}]$, then $\pi[\vec{w}]\in\vec{F}$, since Galois sets are increasing. Hence $\bigvee_{\vec{u}\leq\pi[\vec{w}]}^{\vec{u}\in\vec{F}}\overline{\alpha}_S(\Gamma\vec{w})\subseteq \bigvee_{\vec{w}\in\pi^{-1}[\vec{F}]}\overline{\alpha}_S(\Gamma\vec{w})$. Conversely, if $\vec{w}\in\pi^{-1}[\vec{F}]$, let $\vec{u}=\pi[\vec{w}]$, so that $\vec{u}\in\vec{F}$ and $\vec{u}\leq\pi[\vec{w}]$, which shows that the converse inclusion also holds.
\end{proof}

\begin{lemma}\rm
\label{condition for frame morphisms lemma}
Assuming the frame axioms of Table \ref{frame axioms table}, the condition in the statement of Proposition \ref{condition for frame morphisms prop} can be replaced by the requirement in equation \eqref{equation for morphisms}, which is equivalent to condition  \eqref{equiv condition for morphisms}
\begin{eqnarray}
\pi^{-1}\alpha_R(\Gamma \vec{u})&=&\alpha_S(\pi^{-1}[\Gamma\vec{u}])\label{equation for morphisms}\\
\pi(v)R\vec{u} &\mbox{ iff }& \exists\vec{w}(\vec{u}\leq\pi[\vec{w}]\;\wedge\; vS\vec{w})\label{equiv condition for morphisms}
\end{eqnarray}
\end{lemma}
\begin{proof}
  By Lemma \ref{alpha bar on closed} for any relation $R$ in the frame it holds that $\overline{\alpha}_R(\Gamma\vec{u})=R\vec{u}=\alpha_R(\Gamma\vec{u})$. It follows that if equation \eqref{equation for morphisms} holds, then $\alpha_S(\pi^{-1}[\Gamma\vec{u}])$ is a Galois set, hence $\alpha_S(\pi^{-1}[\Gamma\vec{u}])=\overline{\alpha}_S(\pi^{-1}[\Gamma\vec{u}])$. Hence the identity in \eqref{equation for morphisms} implies the hypothesis in the statement of Proposition \ref{condition for frame morphisms prop}.

  Since $\alpha_S(\pi^{-1}[\Gamma\vec{u}])=\bigcup_{\vec{w}\in\pi^{-1}[\Gamma\vec{u}]}\alpha_S(\Gamma\vec{w})= \bigcup_{\vec{u}\leq\pi[\vec{w}]}\alpha_S(\Gamma\vec{w})=\bigcup_{\vec{u}\leq\pi[\vec{w}]}S\vec{w}$ and $v\in\pi^{-1}\alpha_R(\Gamma \vec{u})$ iff $\pi(v)R\vec{u}$ it follows that the two conditions in \eqref{equation for morphisms} and \eqref{equiv condition for morphisms} are equivalent.
\end{proof}

We conclude with the definition of the category of $\tau$-frames, for a similarity type (sequence of distribution types) $\tau$.

\begin{defn}\rm\label{tau-frames defn}
The objects $\mathfrak{F}=(X,I,Y,(R_\sigma)_{\sigma\in\tau})$ of the category {\bf SRF}$_\tau$ of $\tau$-frames are sorted residuated frames (equivalently, polarities) with a relation of sort type $\sigma$, for each $\sigma$ in $\tau$, subject to axioms FAx1-FAx4 of Table \ref{frame axioms table}. Its morphisms $\pi=(p,q):(X_2,I_2,Y_2,(S_\sigma)_{\sigma\in \tau})\lra(X_1,I_1,Y_1,(R_\sigma)_{\sigma\in \tau})$ are the weak bounded morphisms specified in Definition \ref{sorted weak bounded morphisms} that, in addition, satisfy condition \eqref{equiv condition for morphisms} (axiom MAx4). Table \ref{axioms for the category of frames} collects together all axioms.
\end{defn}

Hereafter, by a weak bounded morphism we shall always mean that condition \eqref{equiv condition for morphisms} (axiom MAx4) is also satisfied.

\begin{table}[t]
\label{axioms for the category of frames}
\caption{Axioms for the frame Category {\bf SRF}$_\tau$}
\begin{enumerate}
  \item[FAx1)] The frame is separated
  \item[FAx2)] For each $\sigma=(\vec{i_j};i_{n+1})$ in the similarity type $\tau$, each $\vec{u}\in\prod_{j=1}^{j=n}Z_{i_j}$, $R_\sigma\vec{u}$ is a closed element of $\mathcal{G}(Z_{i_{n+1}})$
  \item[FAx3)] For each $\sigma=(\vec{i_j};i_{n+1})$ in the similarity type $\tau$, each $w\in Z_{i_{n+1}}$, the $n$-ary relation $wR_\sigma$ is decreasing in every argument place
  \item[FAx4)] All sections of the Galois dual relations $R'_\sigma$ of $R_\sigma$, for each $\sigma$ in $\tau$, are Galois sets\\[1mm]
  For $\pi=(p,q):(X_2,I_2,Y_2,(S_\sigma)_{\sigma\in \tau})\lra(X_1,I_1,Y_1,(R_\sigma)_{\sigma\in \tau})$
  \item[MAx1)] $\forall x'\in X_2\forall y'\in Y_2\;(x'I_2y'\lra p(x')I_1q(y'))$
  \item[MAx2)] $\forall x\in X_1\forall y'\in Y_2(xI_1 q(y')\lra\exists x'\in X_2(x\leq p(x')\wedge x'I_2y')  )$
  \item[MAx3)] $\forall x'\in X_2\forall y\in Y_1(p(x')I_1y\lra\exists y'\in Y_2(y\leq q(y')\wedge x'I_2y')$
  \item[MAx4)] for all $\vec{u}$ and $v$, $\pi(v)R_\sigma\vec{u}$ iff there exists $\vec{w}$ such that $\vec{u}\leq\pi[\vec{w}]$ and $vS_\sigma\vec{w})$
\end{enumerate}

Note that axiom MAx4 assumes proper sorting of $\vec{u}, v$ and $\vec{w}$ as necessitated by the sorting of the relations. We also point out that the axiomatization of the category {\bf SRF}$_\tau$ will be strengthened in Section \ref{stone section} for the purpose of deriving a Stone duality result.
\end{table}

\section{Dual Sorted, Residuated Frames of NLE's}
\label{representation section}
A bounded lattice expansion is a structure $\mathcal{L}=(L,\leq,\wedge,\vee,0,1,\mathcal{F}_1,\mathcal{F}_\partial)$, where $\mathcal{F}_1$ consists of normal lattice operators $f$ of distribution type $\delta(f)=(i_1,\ldots,i_n;1)$ (i.e. of output type 1), while $\mathcal{F}_\partial$  consists of normal lattice operators $h$ of distribution type $\delta(h)=(t_1,\ldots,t_n;\partial)$ (i.e. of output type $\partial$). For representation purposes, nothing depends on the size of the operator families $\mathcal{F}_1$ and $\mathcal{F}_\partial$ and we may as well assume that they contain a single member, say $\mathcal{F}_1=\{f\}$ and $\mathcal{F}_\partial=\{h\}$. In addition, nothing depends on the arity of the operators, so we may assume they are both $n$-ary.

\subsection{Canonical Lattice Frame Construction}
\label{canonical frame section}
The canonical frame is constructed as follows, based on \cite{iulg,sdl,dloa,sdl-exp} (reviewed in Section \ref{normal lat exp section}).\\

First, the base polarity $\mathfrak{F}=(\filt(\mathcal{L}),\upv,\idl(\mathcal{L}))$ consists of the sets $X=\filt(\mathcal{L})$ of filters and $Y=\idl(\mathcal{L})$ of ideals of the lattice and the relation $\upv\;\subseteq\filt(\mathcal{L})\times\idl(\mathcal{L})$ is defined by $x\upv y$ iff $x\cap y\neq\emptyset$, while the representation map $\zeta_1$ sends a lattice element $a\in L$ to the set of filters that contain it, $\zeta_1(a)=\{x\in X\midsp a\in x\}=\{x\in X\midsp x_a\subseteq x\}=\Gamma x_a$. Similarly, a co-represenation map $\zeta_\partial$ is defined by $\zeta_\partial(a)=\{y\in Y\midsp a\in y\}=\{y\in Y\midsp y_a\subseteq y\}=\Gamma y_a$. It is easily seen that $(\zeta_1(a))'=\zeta_\partial(a)$ and, similarly, $(\zeta_\partial(a))'=\zeta_1(a)$. The images of $\zeta_1,\zeta_\partial$ are precisely the families (sublattices of $\gpsi,\gphi$, respectively) of clopen elements of $\gpsi,\gphi$, since clearly $\Gamma x_a={}\rperp\{y_a\}$ and $\Gamma y_a=\{x_a\}\rperp$. For further details the reader is referred to \cite{iulg,sdl}.

Second, for each normal lattice operator a relation is defined, such that if $\delta=(i_1,\ldots,i_n;i_{n+1})$ is the distribution type of the operator, then $\sigma=(i_{n+1};i_1\cdots i_n)$ is the sort type of the relation. Without loss of generality, we have restricted to the families of operators $\mathcal{F}_1=\{f\}$ and $\mathcal{F}_\partial=\{h\}$, so that we shall define two corresponding relations $R,S$ of respective sort types $\sigma(R)=(1;i_1\cdots i_n)$ and $\sigma(S)=(\partial;t_1\cdots t_n)$, where for each $j$, $i_j$ and $t_j$ are in $\{1,\partial\}$. In other words
$R\;\subseteq\; X\times\prod_{j=1}^{j=n}Z_{i_j}$ and $S\;\subseteq\; Y\times \prod_{j=1}^{j=n}Z_{t_j}$.
To define the relations, we use the point operators introduced in \cite{dloa} (see also \cite{sdl-exp}). In the generic case we examine, we need to define two sorted operators
\begin{eqnarray*}
\widehat{f}&:\prod_{j=1}^{j=n}Z_{i_j}\lra Z_1\hskip0.8cm
\widehat{h}&: \prod_{j=1}^{j=n}Z_{t_j}\lra Z_\partial\hskip0.8cm(\mbox{recall that }Z_1=X, Z_\partial=Y)
\end{eqnarray*}
Assuming for the moment that the point operators have been defined, the canonical relations $R,S$ are defined by
\begin{eqnarray}
xR\vec{u} & \mbox{ iff }& \widehat{f}(\vec{u})\subseteq x\;\; (\mbox{for }\; x\in X\;\mbox{ and }\;\vec{u}\in \prod_{j=1}^{j=n}Z_{i_j})\nonumber\\
yS\vec{v} &\mbox{ iff }& \widehat{h}(\vec{v})\subseteq y\;\; (\mbox{for }\; y\in Y\;\mbox{ and }\;\vec{v}\in \prod_{j=1}^{j=n}Z_{t_j})\label{canonical relations defn}
\end{eqnarray}
Returning to the point operators and letting $x_e,y_e$ be the principal filter and principal ideal, respectively, generated by a lattice element $e$, these are uniformly defined as follows, for $\vec{u}\in \prod_{j=1}^{j=n}Z_{i_j}$ and $\vec{v}\in \prod_{j=1}^{j=n}Z_{t_j})$
\begin{eqnarray}
  \widehat{f}(\vec{u}) \;=\; \bigvee\{x_{f(\vec{a})}\midsp\vec{a}\in\vec{u}\}\hskip1.5cm
  \widehat{h}(\vec{v}) \;=\; \bigvee\{y_{h(\vec{a})}\midsp\vec{a}\in\vec{v}\}\label{canonical point operators defn}
\end{eqnarray}
In other words, $\widehat{f}(\vec{u})=\langle\{f(\vec{a})\midsp \vec{a}\in\vec{u}\}\rangle$ is the filter generated by the set $\{f(\vec{a})\midsp \vec{a}\in\vec{u}\}$ and similarly $\widehat{h}(\vec{v})$ is the ideal generated by the set $\{h(\vec{a})\midsp \vec{a}\in\vec{v}\}$.

\begin{ex}[FL$_{ew}$]\rm
\label{example and strategy ex 3}
We consider as an example the case of associative, commutative, integral residuated lattices $\mathcal{L}=(L,\leq,\wedge,\vee,0,1,\circ,\ra)$, the algebraic models of {\bf FL}$_{ew}$ (the associative full Lambek calculus with exchange and weakening), also referred to in the literature as full {\bf BCK}. By residuation of $\circ,\ra$, the distribution types of the operators are $\delta(\circ)=(1,1;1)$ and $\delta(\ra)=(1,\partial;\partial)$. Let $(\filt(\mathcal{L}),\upv,\idl(\mathcal{L}))$ be the canonical frame of the bounded lattice $(L,\leq,\wedge,\vee,0,1)$. Designate the corresponding canonical point operators by $\overt$ and $\leadsto$, respectively. They are defined by \eqref{canonical point   operators defn}
\begin{eqnarray*}
x\overt z &=& \bigvee\{x_{a\circ c}\midsp a\in x\;\wedge\; c\in z\}\in\filt(\mathcal{L})\hskip1cm(x,z\in\filt(\mathcal{L}))\\
x\leadsto v &=&\bigvee\{y_{a\ra c}\midsp a\in x\;\wedge\; c\in v\}\in\idl(\mathcal{L})\hskip9.3mm (x\in\filt(\mathcal{L}),v\in\idl(\mathcal{L}))
\end{eqnarray*}
where recall that we write $x_e,y_e$ for the principal filter and ideal, respectively, generated by the lattice element $e$, so that $x\overt z\in \filt(\mathcal{L})$, while $(x\leadsto v)\in \idl(\mathcal{L})$.

The relations $R^{111},S^{\partial 1\partial}$ are then defined by
\[
uR^{111}xz\;\mbox{ iff }\; x\overt z\subseteq u\hskip2cm yS^{\partial 1\partial}xv\;\mbox{ iff }\;(x\leadsto v)\subseteq y
\]
of sort types $\sigma(R)=(1;11)$ and $\sigma(S)=(\partial;1\partial)$. The canonical {\bf FL}$_{ew}$-frame is therefore the structure $\mathfrak{F}=(\filt(\mathcal{L}),\upv,\idl(\mathcal{L}),R^{111}, S^{\partial 1\partial})$.
\end{ex}

\subsection{Properties of the Canonical Frame}
We first verify that axioms FAx1-FAx3 of Table \ref{axioms for the category of frames} hold for the canonical sorted residuated frame (polarity).
\label{canonical-properties}
\begin{lemma}\rm
\label{elementary props in canonical}
The following hold for the canonical frame.
\begin{enumerate}
\item The frame is separated
\item For $\vec{u}\in \prod_{j=1}^{j=n}Z_{i_j}$ and $\vec{v}\in \prod_{j=1}^{j=n}Z_{t_j}$ the sections $R\vec{u}$ and $S\vec{v}$ are closed elements of $\gpsi$ and $\gphi$, respectively
\item For $x\in X, y\in Y$, the $n$-ary relations $xR, yS$ are decreasing in every argument place
\end{enumerate}
\begin{proof}
  \mbox{}
For 1),
just note that the ordering $\preceq$ is set-theoretic inclusion (of filters, and of ideals, respectively), hence separation of the frame is immediate.

For 2),  by the definition of the relations,  $R\vec{u}=\{x\midsp \widehat{f}(\vec{u})\subseteq x\}=\Gamma(\widehat{f}(\vec{u}))$ is a closed element of $\gpsi$ and similarly for $S\vec{v}$.

For 3),  if $w\subseteq u_k$, then $\{x_{f(a_1,\ldots,a_n)}\midsp a_k\in w\wedge \bigwedge_{j\neq k}(a_j\in u_j)\}$ is a subset of the set $\{x_{f(a_1,\ldots,a_n)}\midsp  \bigwedge_{j}(a_j\in u_j)\}$, hence taking joins it follows that $\widehat{f}(\vec{u}[w]_k)\subseteq\widehat{f}(\vec{u})$. By definition, if $xR\vec{u}$ holds, then we obtain $\widehat{f}(\vec{u}[w]_k)\subseteq\widehat{f}(\vec{u})\subseteq x$, hence $xR\vec{u}[w]_k$ holds as well. Similarly for the relation $S$.
\end{proof}
\end{lemma}

\begin{lemma}\rm\label{unified relational}
In the canonical frame, $xR\vec{u}$ holds iff $\forall\vec{a}\in L^n\;(\vec{a}\in\vec{u}\lra f(\vec{a})\in x)$. Similarly, $yS\vec{v}$ holds iff $\forall\vec{a}\in L^n\;(\vec{a}\in\vec{v}\lra h(\vec{a})\in y)$.
\end{lemma}
\begin{proof}
 By definition $xR\vec{u}$ holds iff $\widehat{f}(\vec{u})\subseteq x$, where $\widehat{f}(\vec{u})$, by its definition \eqref{canonical point operators defn} is the filter generated by the elements $f(\vec{a})$, for $\vec{a}\in\vec{u}$, hence clearly  $\vec{a}\in\vec{u}$ implies $f(\vec{a})\in x$. Similarly for the relation $S$.
\end{proof}

The next two lemmas were first stated and proven in \cite{kata2z}, but we present proofs here as well, to keep this article self-contained.
\begin{lemma}
\rm\label{same relation}
Where $R',S'$ are the Galois dual relations of the canonical relations $R,S$, $yR'\vec{u}$ holds iff $\widehat{f}(\vec{u})\upv y$ iff $\exists\vec{b}(\vec{b}\in\vec{u}\wedge f(\vec{b})\in y)$. Similarly, $xS'\vec{v}$ holds iff $x\upv \widehat{h}(\vec{v})$ iff $\exists\vec{e}(\vec{e}\in\vec{v}\wedge h(\vec{e})\in x)$.
\end{lemma}
\begin{proof}
  By definition of the Galois dual relation, $yR'\vec{u}$ holds iff for all $x\in X$, if $xR\vec{u}$ obtains, then $x\upv y$. By definition of the canonical relation $R$, for any $x\in X$, $xR\vec{u}$ holds iff $\widehat{f}(\vec{u})\subseteq x$ and thereby $\widehat{f}(\vec{u})R\vec{u}$ always obtains. Hence, $yR'\vec{u}$ is equivalent to $\forall x\in X\;(\widehat{f}(\vec{u})\subseteq x\lra x\cap y\neq\emptyset)$, from which it follows that $\widehat{f}(\vec{u})\upv y$ iff $yR'\vec{u}$ obtains.

To show that $yR'\vec{u}$ holds  iff $\exists\vec{a}(\vec{a}\in\vec{u}\wedge f(\vec{a})\in y)$, since the direction from right to left is trivially true, assume $yR'\vec{u}$, or, equivalently by the argument given above, assume that $\widehat{f}(\vec{u})\upv y$, i.e. $\widehat{f}(\vec{u})\cap y\neq \emptyset$ and let $e\in \widehat{f}(\vec{u})\cap y$. By $e\in\widehat{f}(\vec{u})$ and definition of $\widehat{f}(\vec{u})$ as the filter generated by the set $\{f(\vec{a})\midsp\vec{a}\in\vec{u}\}$, let $\vec{a}^1,\ldots,\vec{a}^s$, for some positive integer $s$, be $n$-tuples of lattice elements (where $\vec{a}^r=(a^r_1,\ldots,a^r_n)$, for $1\leq r\leq s$) such that $f(\vec{a}^1)\wedge\cdots\wedge f(\vec{a}^s)\leq e$ and $a^r_j\in u_j$ for each $1\leq r\leq s$ and $1\leq j\leq n$. Recall that the distribution type of $f$ is $\delta(f)=(i_1,\ldots,i_n;1)$, where for $j=1,\ldots,n$ we have $i_j\in\{1,\partial\}$ and define elements $b_1,\ldots,b_n$ as follows.
\[
b_j=\left\{
\begin{array}{cl}
a^1_j\wedge\cdots\wedge a^s_j &\;\mbox{ if }\; i_j=1\\[1mm]
a^1_j\vee\cdots\vee a^s_j &\;\mbox{ if }\; i_j=\partial
\end{array}
\right.
\]
When $i_j=1$, $f$ is monotone at the $j$-th argument place, $u_j$ is a filter and $b_j\leq a^r_j\in u_j$, for all $r=1,\ldots,s$, so that $b_j=a^1_j\wedge\cdots\wedge a^s_j\in u_j$. Similarly, when $i_{j'}=\partial$, $f$ is antitone at the $j'$-th argument place, while $u_{j'}$ is an ideal, so that $b_{j'}=a^1_{j'}\vee\cdots\vee a^s_{j'}\in u_{j'}$. This shows that $\vec{b}\in\vec{u}$ and it remains to show that $f(\vec{b})\in y$. We argue that $f(\vec{b})\leq f(\vec{a}^1)\wedge\cdots\wedge f(\vec{a}^s)\leq e$ and the desired conclusion follows by the fact that $e\in y$, an ideal.

For any $1\leq r\leq s$, let $\vec{a}^r[b_j]_j^{i_j=1}$ be the result of replacing $a^r_j$ by $b_j$ in the tuple $\vec{a}^r$ and in every position $j$ from 1 to $n$ such that $i_j=1$ in the distribution type of $f$. Since $b_j\leq a^r_j$ and $f$ is monotone at any such $j$-th argument place, it follows that $f(\vec{a}^r[b_j]_j^{i_j=1})\leq f(\vec{a}^r)$, for all $1\leq r\leq s$.

In addition, for any $1\leq r\leq s$, let $\vec{a}^r[b_j]_j^{i_j=1}[b_{j'}]_{j'}^{i_{j'}=\partial}$ be the result of replacing $a^r_{j'}$ by $b_{j'}$ in the tuple $\vec{a}^r[b_j]_j^{i_j=1}$ and in every position $j'$ from 1 to $n$ such that $i_{j'}=\partial$ in the distribution type of $f$. Since $b_{j'}\geq a^r_{j'}$ and $f$ is antitone at any such $j'$-th argument place, it follows that $f(\vec{a}^r[b_j]_j^{i_j=1}[b_{j'}]_{j'}^{i_{j'}=\partial})\leq f(\vec{a}^r[b_j]_j^{i_j=1})\leq f(\vec{a}^r)$, for all $1\leq r\leq s$. Since $\vec{a}^r[b_j]_j^{i_j=1}[b_{j'}]_{j'}^{i_{j'}=\partial}=\vec{b}$ we obtain  that
\[
f(\vec{b})=f(\vec{a}^r[b_j]_j^{i_j=1}[b_{j'}]_{j'}^{i_{j'}=\partial})\leq f(\vec{a}^r[b_j]_j^{i_j=1})\leq f(\vec{a}^1)\wedge\cdots\wedge f(\vec{a}^s)\leq e
\]
hence $f(\vec{b})\in y$ and  this completes the proof, as far as the relation $R$ is concerned.

The argument for the relation $S$ is similar and can be safely left to the reader.
\end{proof}

We can now prove that the frame axiom FAx4 of Table \ref{axioms for the category of frames} also holds in the canonical frame.
\begin{lemma}\rm
\label{canonical frame properties prop}
 In the canonical frame, all sections of the Galois dual relations $R', S'$ of the canonical relations $R,S$ are Galois sets.
\end{lemma}
\begin{proof}
There are two cases to handle, one for each  of the relations $R',S'$, with two subcases for each one, depending on whether $i_k$ is $1$, or $\partial$. The cases of the two relations are similar, we have presented the proof for the relation $R'$ in \cite{kata2z} and we detail here the other case, for the relation $S'$.\\

Case of the relation $S'$:

The section $S'\vec{v}=(S\vec{v})^\prime$ is a Galois (stable) set, by its definition.  Recall that the sort type of $S$ is $\sigma(S)=(\partial;t_1\cdots t_n)$, where $t_k\in\{1,\partial\}$ for each $1\leq k\leq n$, and that $S$ was defined given the lattice normal operator $h$, of distribution type $\delta(h)=(t_1,\ldots,t_n;\partial)$.

Let now $x\in X$ and consider any section $xS'\vec{v}[\_]_k$. We again distinguish the subcases $i_k=1$, or $i_k=\partial$. When $i_k=\partial$ (same as the output type of $h$), then $h$ is monotone at the $k$-th argument place and it distributes over finite lattice meets, whereas when $i_k=1$, then $h$ is antitone at the $k$-th argument place and it co-distributes over finite lattice joins, turning them to meets. Furtheremore, by Lemma \ref{same relation} of this proposition, $xS'\vec{y}$ holds iff $x\upv \widehat{h}(\vec{y})$ iff $\exists\vec{e}(\vec{e}\in\vec{y}\wedge h(\vec{e})\in x)$.\\

Subcase $i_k=\partial$):\\
Then $xS'\vec{y}[\_]_k\subseteq Y=\idl(\mathcal{L})$ and note that the output type of $h$ is also $t_{n+1}=\partial$.

Let $W=\{b\in L\midsp\exists \vec{e}[\_]_k\in \vec{y}[\_]_k\;\;h(\vec{e}[b]_k)\in x\}$ and $w$ be the filter generated by $W$. If $v$ is an ideal such that $xS'\vec{y}[v]_k$ holds, then by Lemma \ref{same relation} $x\upv \widehat{h}(\vec{y}[v]_k)$, equivalently, for some tuple of lattice elements $\vec{e}[b]_k\in\vec{y}[v]_k$ we have $h(\vec{e}[b]_k)\in x$. Then $b\in w\cap v$, i.e. $w\upv v$ and then $w\upv xS'\vec{y}[\_]_k$.

Let now $q$ be an ideal $q\in(xS'\vec{y}[\_]_k)^{\prime\prime}$. We show that $xS'\vec{y}[q]_k$ holds.

By the assumption on $q$ and the fact that $w\upv xS'\vec{y}[\_]_k$ we obtain $w\upv q$, i.e. there is some element $b\in w\cap q\neq\emptyset$.
By definition of $w$, there exist lattice elements $b_1,\ldots,b_s\in W$, for some positive integer $s$, such that $b_1\wedge\cdots\wedge b_s\leq b$. Since $b_r\in W$, for $1\leq r\leq s$, there exist tuples of lattice elements $\vec{c}^r[\_]_k=(c^r_1,\ldots,c^r_{k-1},\_,c^r_{k+1},\ldots,c^r_n)$, for $1\leq r\leq s$,  such that $h(\vec{c}^r[b_r]_k)\in x$. Define
\[
e_j=\left\{
\begin{array}{cl}
c^1_j\wedge\cdots\wedge c^s_j &\;\mbox{ if }\; t_j=1\\[1mm]
c^1_j\vee\cdots\vee c^s_j &\;\mbox{ if }\; t_j=\partial
\end{array}
\right.
\]
Considering the monotonicity properties of $h$ and using the notation introduced in the previous cases, observe that, for each $1\leq r\leq s$ we have
\[
h(\vec{e}[b_r]_k)=h(\vec{c}^r[b_r]_k[e_j]_j^{t_j=\partial}[e_{j'}]_{j'}^{t_{j'}=1})\geq h(\vec{c}^r[b_r]_k)\in x\in\filt(\mathcal{L})
\]
and so $\bigwedge_{r=1}^{r=s}h(\vec{e}[b_r]_k)\in x$. By the case assumption, $h$ is monotone and it distributes over finite meets at the $k$-th argument place, hence we obtain that
\[
h(\vec{e}[b]_k)\geq h(\vec{e}[b_1\wedge\cdots\wedge b_s]_k)=\bigwedge_{r=1}^{r=s}h(\vec{e}[b_r]_k)\in x
\]
By Lemma \ref{same relation}, given $\vec{e}[b]_k\in\vec{v}[q]_k$ and $h(\vec{e}[b]_k)\in x$ we conclude that $xS'\vec{v}[q]_k$ holds and this proves that the section $xS'\vec{v}[\_]_k$ is a Galois (co-stable) set.
\\

Subcase $i_k=1$):\\
Then $xS'\vec{y}[\_]_k\subseteq X=\filt(\mathcal{L})$.

Let $W=\{b\in L\midsp\exists \vec{e}[\_]_k\in \vec{y}[\_]_k\;\;h(\vec{e}[b]_k)\in x\}$ and $v$ be the ideal generated by $W$. If $z$ is any filter such that $xS'\vec{y}[z]_k$, then by Lemma \ref{same relation} there is a tuple of lattice elements $\vec{e}[b]_k\in\vec{y}[z]_k$ such that $h(\vec{e}[b]_k)\in x$. Thus $b\in z\cap v$ and this shows that $xS'\vec{y}[\_]_k\upv v$.

We now assume that $z\in(xS'\vec{y}[\_]_k)^{\prime\prime}$ and show that $xS'\vec{y}[z]_k$ holds.

By $z\in(xS'\vec{y}[\_]_k)^{\prime\prime}$ and $xS'\vec{y}[\_]_k\upv v$ it follows that $z\upv v$, so for some lattice element $b$ we have $b\in z\cap v\neq\emptyset$.

By definition of the ideal $v$, there exist elements $b_1,\ldots,b_s\in W$, for some positive integer $s$, such that $b\leq b_1\vee\cdots\vee b_s$. By definition of $W$ there exist tuples of lattice elements $\vec{c}^r[\_]_k$, with $1\leq r\leq s$, such that $h(\vec{c}^r[b_r]_k)\in x$ for each $1\leq r\leq s$. Define
\[
e_j=\left\{
\begin{array}{cl}
c^1_j\wedge\cdots\wedge c^s_j &\;\mbox{ if }\; t_j=1\\[1mm]
c^1_j\vee\cdots\vee c^s_j &\;\mbox{ if }\; t_j=\partial
\end{array}
\right.
\]
Considering the monotonicity properties of $h$ and using the notation introduced in the previous cases, observe that, for each $1\leq r\leq s$ we have
\[
h(\vec{e}[b_r]_k)=h(\vec{c}^r[b_r]_k[e_j]_j^{t_j=\partial}[e_{j'}]_{j'}^{t_{j'}=1})\geq h(\vec{c}^r[b_r]_k)\in x\in\filt(\mathcal{L})
\]
and so $\bigwedge_{r=1}^{r=s}h(\vec{e}[b_r]_k)\in x$. By the case assumption, $h$ is antitone and it co-distributes over finite joins at the $k$-th argument place, turning them to meets, hence we obtain that
\[
h(\vec{e}[b]_k)\geq h(\vec{e}[b_1\vee\cdots\vee b_s]_k)=\bigwedge_{r=1}^{r=s}h(\vec{e}[b_r]_k)\in x
\]
By Lemma \ref{same relation}, given $\vec{e}[b]_k\in\vec{y}[z]_k$ and $h(\vec{e}[b]_k)\in x$, we conclude that $xS'\vec{y}[z]_k$ holds and this shows that the section $xS'\vec{y}[\_]_k$ is a Galois (stable) set.
\end{proof}

The canonical frame for a lattice expansion $\mathcal{L}=(L,\leq,\wedge,\vee,0,1,f,h)$, where $\delta(f)=(i_1,\ldots,i_n;1)$ and $\delta(h)=(t_1,\ldots,t_n;\partial)$ ($i_j,t_j\in\{1,\partial\}$) is  the structure $\mathcal{L}_+=\mathfrak{F}=(\filt(\mathcal{L}),\upv,\idl(\mathcal{L}),R,S)$. By Proposition \ref{canonical frame properties prop}, the canonical relations $R,S$ are compatible with the Galois connection generated by $\upv\;\subseteq X\times Y$, in the sense that all sections of their Galois dual relations are Galois sets.
Set operators $\alpha_R, \eta_S$ are defined as in Section \ref{frames section} and we let $\overline{\alpha}_R,\overline{\eta}_S$ be the closures of their restrictions to Galois sets (according to their distribution types). Note that $\overline{\alpha}_R(\vec{F})\in\gpsi$, while $\overline{\eta}_S(\vec{G})\in\gphi$, given the output types of $f,h$ (alternatively, given the sort types of $R,S$).

It follows from Theorem \ref{distribution from stability thm} and Lemma \ref{canonical frame properties prop}, that the sorted operators $\overline{\alpha}_R,\overline{\eta}_S$ on Galois sets distribute over arbitrary joins of Galois sets (stable or co-stable, according to the sort types of $R,S$) in each argument place.

Note that $\overline{\alpha}_R,\overline{\eta}_S$ are sorted maps, taking their values in $\gpsi$ and $\gphi$, respectively. We define single-sorted maps on $\gpsi$ (analogously for $\gphi$) by composition with the Galois connection
\begin{eqnarray}
  \overline{\alpha}_f(A_1,\ldots,A_n) &=& \overline{\alpha}_R(\ldots,\underbrace{A_j}_{i_j=1},\ldots,\underbrace{A^\prime_r}_{i_r=\partial},\ldots) \hskip5mm(A_1,\ldots,A_n\in\gpsi)\label{single-sorted f} \\
  \overline{\eta}_h(B_1,\ldots,B_n) &=& \overline{\eta}_R(\ldots,\underbrace{B_r}_{i_r=\partial},\ldots,\underbrace{B^\prime_j}_{i_j=1},\ldots) \hskip5mm(B_1,\ldots,B_n\in\gphi)\label{single-sorted h}
\end{eqnarray}
Given that the Galois connection is a duality of Galois stable and Galois co-stable sets, it follows that the distribution type of $\overline{\alpha}_f$ is that of $f$ and that $\overline{\alpha}_f$ distributes, or co-distributes, over arbitrary joins and meets in each argument place, according to its distribution type, returning joins in $\gpsi$. Similarly, for $\overline{\eta}_h$.
Thus, the lattice representation maps $\zeta_1:(L,\leq,\wedge,\vee,0,1)\lra\gpsi$ and $\zeta_\partial:(L,\leq,\wedge,\vee,0,1)\lra\gphi$ are extended to maps $\zeta_1:\mathcal{L}\lra\gpsi$ and $\zeta_\partial:\mathcal{L}\lra \gphi$ by setting
\begin{eqnarray}
\zeta_1(f(a_1,\ldots,a_n))&=&\overline{\alpha}_f(\zeta_1(a_1),\ldots, \zeta_1(a_n))   
    \hskip2mm=\; \overline{\alpha}_R(\ldots,\underbrace{\zeta_1(a_j)}_{i_j=1},\ldots, \underbrace{\zeta_\partial(a_r)}_{i_r=\partial},\ldots)\nonumber\\
\zeta_\partial(f(a_1,\ldots,a_n)) &=& \left(\;\overline{\alpha}_f(\zeta_1(a_1),\ldots, \zeta_1(a_n))\;\right)^\prime
\label{canonical rep of lat ops f}\\[4mm]
\zeta_1(h(a_1,\ldots,a_n))&=& \left(\;\overline{\eta}_h(\zeta_\partial(a_1),\ldots,\zeta_\partial(a_n))\;\right)^\prime \nonumber\\
\zeta_\partial(h(a_1,\ldots,a_n))&=& \overline{\eta}_h(\zeta_\partial(a_1),\ldots,\zeta_\partial(a_n))\label{canonical rep of lat ops h}
\end{eqnarray}

\subsection{The Duals of Lattice Expansion Homomorphisms}
\label{homos section}
With the next Proposition, we verify that the frame axioms MAx1-MAx4 hold in the canonical frame construction, as well.

\begin{prop}\rm
Let $h:\mathcal{L}\lra\mathcal{L}^*$ be a homomorphism of normal lattice expansions of similarity type $\tau$, with corresponding normal operators $f_\sigma, f^*_\sigma$, for each $\sigma\in\tau$, and
$\mathfrak{F}=(X,I,Y,(R_\sigma)_{\sigma\in\tau})$, $\mathfrak{F}^*=(X^*,I^*,Y^*,(S_\sigma)_{\sigma\in\tau})$ their canonical dual frames, as defined in Section \ref{canonical frame section}. For $x^*\in X^*$ and $y^*\in Y^*$ define
 $p(x^*)=h^{-1}[x^*]=\{a\in\mathcal{L}\midsp h(a)\in x^*\}$, $q(y^*)=h^{-1}[y^*]=\{a\in\mathcal{L}\midsp h(a)\in y^*\}$. Then  $\pi=(p,q):(X^*,I^*,Y^*,(S_\sigma)_{\sigma\in\tau})\lra(X,I,Y,(R_\sigma)_{\sigma\in\tau})$ is a weak bounded morphism of their dual sorted residuated frames.
\end{prop}
\begin{proof}
  For the first condition (relation preservation, axiom MAx1 of Table \ref{axioms for the category of frames}), assume $x^*I^*y^*$, i.e. $x^*\cap y^*=\emptyset$. If $p(x^*)\cap q(y^*)\neq\emptyset$ and $a\in h^{-1}[x^*]\cap h^{-1}[y^*]$, then $h(a)\in x^*\cap y^*\neq\emptyset$, contradicting the hypothesis.

  For the second condition (axiom MAx2 of Table \ref{axioms for the category of frames}), let $x\in X, y^*\in Y^*$ be arbitrary and assume $xI q(y^*)$, which means $x\cap h^{-1}[y^*]=\emptyset$, i.e. $\forall a\in\mathcal{L}(a\in x\lra h(a)\not\in y^*)$.
  Let $x^*$ be the filter generated by the set $h[x]=\{h(a)\midsp a\in x\}$. Then $x\subseteq h^{-1}[x^*]=p(x^*)$. Suppose now that $e\in x^*\cap y^*$. Let $a_1,\ldots,a_n\in x$ such that $h(a_1)\wedge\cdots\wedge h(a_n)\leq e$. Since $h$ is a lattice homomorphism and $x$ is a filter, if $a=a_1\wedge\cdots\wedge a_n$, then $a\in x$ and $h(a_1)\wedge\cdots\wedge h(a_n)=h(a)\leq e\in y^*$, but $y^*$ is an ideal, hence $h(a)\in y^*$, contradiction. Hence it also holds for this $x^*$ that $x^*I^*y^*$.

  The proof of the third condition (axiom MAx3 of Table \ref{axioms for the category of frames}) is similar, but we detail it anyway. So let $x^*\in X^*, y\in Y$, arbitrary, and assume that $p(x^*)I y$, i.e. $h^{-1}[x^*]\cap y=\emptyset$. In other words, $\forall a\in\mathcal{L}(a\in y\lra h(a)\not\in x^*)$. Let $y^*$ be the ideal generated by the set $h[y]=\{h(e)\midsp e\in y\}$. Then $y\subseteq h^{-1}[y^*]=q(y^*)$. If $e\in x^*\cap y^*$, let $a_1,\ldots, a_n\in y$ such that $e\leq h(a_1)\vee\cdots\vee h(a_n)=h(a_1\vee\cdots \vee a_n)$, since $h$ is a homomorphism. Put $a=a_1\vee\cdots \vee a_n\in y$. Then $e\in x^*$, a filter, hence $h(a)\in x^*$, while $a\in y$, contradiction. Hence we obtain $x^*I^*y^*$ for the ideal $y^*$ defined above.

  Let $R_\sigma, S_\sigma$ be the canonical relations in the two frames, of the same sort type $\sigma=(i_{n+1};\vec{i_j})$, corresponding to the normal lattice operators $f_\sigma$ and $f^*_\sigma$. Let $\alpha_R,\alpha_S$ (we drop the subscript $\sigma$, for simplicity) be the induced image operators.

  We show that $\pi^{-1}$ is a homomorphism, i.e. that axiom MAx4 of Table \ref{axioms for the category of frames} holds, or, equivalently by Lemma \ref{condition for frame morphisms lemma}, that $\pi^{-1}\alpha_R(\Gamma \vec{u})=\alpha_S(\pi^{-1}[\Gamma\vec{u}])$.

  We calculate that
  \begin{tabbing}
  $v\in\pi^{-1}\alpha_R(\Gamma\vec{u})$\hskip2mm\= iff\hskip1mm\= $\pi(v)\in\alpha_R(\Gamma\vec{u})$\hskip1cm\=\\
  \>iff\> $h^{-1}[v]\in\alpha_R(\Gamma\vec{u})$ \> By definition of $\pi$ ($v$ a filter, or an ideal)\\
  \>iff\> $h^{-1}[v]\in R\vec{u}$ \> By Lemma \ref{alpha bar on closed} \\
  \>iff\> $h^{-1}[v]\in\Gamma(\widehat{f}\vec{u})$ \> By definition of the canonical relation\\
  \>iff\> $\widehat{f}\vec{u}\leq h^{-1}[v]$ \\
  \>iff\> $\widehat{f}\vec{u}\subseteq h^{-1}[v]$ \> the order is inclusion in canonical frame\\
  \>iff\> $\forall \vec{a}(\vec{a}\in \vec{u}\lra f(\vec{a})\in h^{-1}[v])$ \>\hskip1.5cm definition of $\widehat{f}$\\
  \>iff\> $\forall \vec{a}(\vec{a}\in \vec{u}\lra h(f(\vec{a}))\in v)$\\
  \>iff\> $\forall \vec{a}(\vec{a}\in \vec{u}\lra f^*(h[\vec{a}])\in v)$ \> \hskip1.5cm $h$ is a homomorphism
  \\[2mm]
  \>\> Note that $\vec{w}\in\pi^{-1}[\Gamma\vec{u}]$ iff $\vec{u}\subseteq\pi[\vec{w}]$ and since $\pi^{-1}[\Gamma\vec{u}]$ is a\\
  \>\> Galois set, by Lemma \ref{basic facts} we obtain $\pi^{-1}[\Gamma\vec{u}]=\bigcup_{\vec{u}\subseteq\pi[\vec{w}]}\Gamma\vec{w}$\\
  \>\> hence we compute\\[1mm]
  $\alpha_S\pi^{-1}[\Gamma\vec{u}]$\>=\> $\alpha_S\bigcup_{\vec{u}\subseteq\pi[\vec{w}]}\Gamma\vec{w}$ \>\ $=  \bigcup_{\vec{u}\subseteq\pi[\vec{w}]} \overline{\alpha}_S \Gamma\vec{w}$      \\
  \>=\> $\bigcup_{\vec{u}\subseteq\pi[\vec{w}]}S\vec{w}$ \>\  $=\bigcup_{\vec{u}\subseteq\pi[\vec{w}]}\Gamma(\widehat{f^*}\vec{w})$\\
  \>=\> $\bigcup_{h[\vec{u}]\subseteq\vec{w}}\Gamma(\widehat{f^*}\vec{w})$
  \end{tabbing}

   We first prove that  $\bigcup_{h[\vec{u}]\subseteq\vec{w}}\Gamma(\widehat{f^*}\vec{w})\subseteq \pi^{-1}\alpha_R(\Gamma\vec{u})$. To clarify notation, $h[\vec{u}]$ is the tuple $(h[u_1],\ldots,h[u_n])$, where $h[u_j]=\{h(e)\midsp e\in u_j\}$. Then by $h[\vec{u}]\subseteq\vec{w}$ we mean the (conjunction of the) pointwise inclusions $h[u_j]\subseteq w_j$.

  Let then $\vec{w}$ be such that $h[\vec{u}]\subseteq\vec{w}$ and assume that $\widehat{f^*}\vec{w}\subseteq v$. To show that $v\in \pi^{-1}\alpha_R(\Gamma\vec{u})$ we assume that $\vec{a}\in\vec{u}$ and prove that $f^*(h[\vec{a}])\in v$.

  By $\vec{a}\in\vec{u}$ and $h[\vec{u}]\subseteq\vec{w}$ we obtain $h[\vec{a}]\in\vec{w}$ (i.e. $h(a_j)\in w_j$, for all $j=1,\ldots,n$). By definition in the canonical frame $\widehat{f^*}\vec{w}$ is (the filter, or ideal, depending on the distribution type $\sigma$) generated by the set $\{f^*(\vec{e})\midsp \vec{e}\in\vec{w}\}$. By the hypothesis that $\widehat{f^*}\vec{w}\subseteq v$ we obtain that $\{f^*(\vec{e})\midsp \vec{e}\in\vec{w}\}\subseteq v$. Since $h[\vec{a}]\in\vec{w}$, we have in particular that $f^*(h[\vec{a}])\in v$, q.e.d.

  Conversely, we show that   $\pi^{-1}\alpha_R(\Gamma\vec{u})\subseteq\bigcup_{h[\vec{u}]\subseteq\vec{w}}\Gamma(\widehat{f^*}\vec{w})$.

  Let $v\in\pi^{-1}\alpha_R(\Gamma\vec{u})$ which means, by the calculation above, that $\pi(v)R\vec{u}$, equivalently in the canonical frame it means that $\{f^*(h[\vec{a}])\midsp\vec{a}\in\vec{u}\}\subseteq v$. The claim is that there is a tuple $\vec{w}=(w_1,\ldots,w_n)$, where $w_j$ is a filter if $i_j=1$ in the distribution type of $f,f^*$ and $w_j$ is an ideal when $i_j=\partial$, such that $h[\vec{u}]\subseteq\vec{w}$ and $\widehat{f^*}\vec{w}\subseteq v$.

  Define $w_j=\langle h[u_j]\rangle$ to be the filter of $\mathcal{L}^*$ generated by the set $h[u_j]$, if $i_j=1$ and otherwise let it be the ideal generated by the same set. Thus trivially $h[\vec{u}]\subseteq\vec{w}$ holds. By definition of $\widehat{f^*}$, it suffices to get its set of generators $\{f^*(\vec{e})\midsp \vec{e}\in\vec{w}\}$ to be contained in $v$.

  Let $\vec{e}\in\vec{w}$ so that for $1\leq j\leq n$, $e_j\in w_j$.

 If $i_j=1$, then $w_j$ is the filter generated by the set $h[u_j]$ (where $u_j$ is also a filter) and there are elements $a^j_1,\cdots,a^j_s\in u_j$ such that $ha^j_1\wedge\cdots\wedge ha^j_s\leq e_j$. Since $h$ is a homomorphism and $u_j$ is a filter, letting $a^j=a^j_1\wedge\cdots\wedge a^j_s\in u_j$ we obtain $ha^j\leq e_j$.

 If $i_j=\partial$, then $w_j$ is the ideal generated by the set $h[u_j]$ (where $u_j$ is also an ideal) and there are elements $a^j_1,\cdots,a^j_t\in u_j$ such that $e_j\leq ha^j_1\vee\cdots\vee ha^j_t$. Letting $a^j$ be the disjunction $a^j=a^j_1\vee\cdots\vee a^j_t$ we similarly obtain $e_j\leq ha^j$.

  If the output type $i_{n+1}=1$, then $v$ is a filter, $f,f^*$ are monotone at the $j$-th position, whenever $i_j=1$ and they are antitone at any position $i_{j'}$ with $i_{j'}=\partial$. Therefore, $f^*(ha^1,\ldots,ha^n)\leq f^*(e_1,\ldots,e_n)$, which we may compactly write as $f^*(h[\vec{a}])\leq f^*(\vec{e})$. From the hypothesis on $v$ we have that $f^*(h[\vec{a}])\in v$, which is a filter when $i_{n+1}=1$, hence also $f^*(\vec{e})\in v$.

  If the output type $i_{n+1}=\partial$, then $v$ is an ideal, $f,f^*$ are monotone at the $j$-th position, whenever $i_j=\partial$ and they are antitone at any position $i_{j'}$ with $i_{j'}=1$. Thereby, $f^*(e_1,\ldots,e_n)\leq f^*(ha^1,\ldots,ha^n)$, i.e. $f^*(\vec{e})\leq f^*(h[\vec{a}])\in v$, now an ideal, hence again $f^*(\vec{e})\in v$.
\end{proof}

\section{Stone Duality}
\label{stone section}
The results we have presented can be extended to a Stone duality, by combining with our results in \cite{sdl,sdl-exp}. The functor $\type{F}:{\bf NLE}_\tau\lra{\bf SRF}_\tau^{op}$ sends a normal lattice expansion $\mathfrak{L}$ in {\bf NLE}$_\tau$ to its dual frame $\mathcal{L}_+=\mathfrak{F}$ in {\bf SRF}$_\tau$ (detailed in Section \ref{canonical frame section}) and a lattice
expansion homomorphism $h$ to a weak bounded morphism $\pi$ (detailed in Section \ref{homos section}). Conversely, we have constructed a functor $\type{L}:{\bf SRF}_\tau^{op}\lra{\bf NLE}_\tau$, sending a frame $\mathfrak{F}$ in {\bf SRF}$_\tau$ to its complex algebra $\mathfrak{F}^+$ (Definition \ref{complex algebra defn}), which is a complete normal lattice expansion, and a weak bounded morphism $\pi$ to a complete normal lattice expansion homomorphism $\pi^{-1}$ (detailed in Section \ref{morphisms for relations section}). 

For a full Stone duality, a subcategory {\bf SRF}$^*_\tau$ will be identified, by strengthening the axiomatization of sorted residuated frames with relations. We replace axiom FAx2 with its stronger version FAx2*, we add axioms FAx5-FAx7 from \cite{sdl,sdl-exp} as well as axioms MAx5, MAx6. The full axiomatization of the category {\bf SRF}$^*_\tau$ is presented in Table \ref{full axiomatization}. Call a point $u\in X\cup Y$ {\em clopen} if $\Gamma u$ is clopen, i.e. if there exists a point $v$ of the dual sort  such that $\Gamma u=\{v\}'$.
\begin{table}[htbp]
\label{full axiomatization}
\caption{Axioms for the Subcategory {\bf SRF}$^*_\tau$ of {\bf SRF}$_\tau$}
\begin{enumerate}
  \item[FAx1)] The frame is separated
  \item[FAx2*)] For each $\sigma=(\vec{i_j};i_{n+1})$ in the similarity type $\tau$, each $\vec{u}\in\prod_{j=1}^{j=n}Z_{i_j}$, $R_\sigma\vec{u}$ is a closed element of $\mathcal{G}(Z_{i_{n+1}})$ and if all points $u_j$ are clopen, then $R_\sigma\vec{u}$ is a clopen element of $\mathcal{G}(Z_{i_{n+1}})$
  \item[FAx3)] For each $\sigma=(\vec{i_j};i_{n+1})$ in the similarity type $\tau$, each $w\in Z_{i_{n+1}}$, the $n$-ary relation $wR_\sigma$ is decreasing in every argument place
  \item[FAx4)] All sections of the Galois dual relations $R'_\sigma$ of $R_\sigma$, for each $\sigma$ in $\tau$, are Galois sets
  \item[FAx5)] Clopen sets are closed under finite intersections in each of $\gpsi,\gphi$
  \item[FAx6)] The family of closed sets, for each of $\gpsi,\gphi$, is the intersection closure of the respective set of clopens
  \item[FAx7)] Each of $X,Y$ carries a Stone topology generated by their respective families of clopen sets together with their complements
    \\[2mm]
    For $\pi=(p,q):(X_2,I_2,Y_2,(S_\sigma)_{\sigma\in \tau})\lra(X_1,I_1,Y_1,(R_\sigma)_{\sigma\in \tau})$
  \item[MAx1)] $\forall x'\in X_2\forall y'\in Y_2\;(x'I_2y'\lra \pi(x')I_1\pi(y'))$
  \item[MAx2)] $\forall x\in X_1\forall y'\in Y_2(xI_1 \pi(y')\lra\exists x'\in X_2(x\leq \pi(x')\wedge x'I_2y')  )$
  \item[MAx3)] $\forall x'\in X_2\forall y\in Y_1(\pi(x')I_1y\lra\exists y'\in Y_2(y\leq \pi(y')\wedge x'I_2y')$
  \item[MAx4)] for all $\vec{u}$ and $v$, $\pi(v)R_\sigma\vec{u}$ iff there exists $\vec{w}$ s.t. $\vec{u}\leq\pi[\vec{w}]$ and $vS_\sigma\vec{w}$
  \item[MAx5)] for all points $u$, $\pi^{-1}(\Gamma u)=\Gamma v$, for some (unique, by separation) $v$
  \item[MAx6)] for all points $u$, if $\Gamma u=\{v\}'$, for some (unique, by separation) $v$, then\\
   \hspace*{6mm}$\pi^{-1}(\Gamma u)=\left(\pi^{-1}(\Gamma v)\right)^\prime$  
   \\[2mm]
   We have glossed over sorts in order to get simple statements, which can easily be expressed in first-order terms. For example, MAx5 is the conjunction of the following two first-order statements (and similarly for the axiom MAx6).
\begin{description}
\item $\forall u\in X_1\exists v\in X_2 \forall w\in X_2\;(u\leq p(w)\longleftrightarrow v\leq w)$
\item $\forall u\in Y_1\exists v\in Y_2\forall w\in Y_2\;(u\leq q(w)\longleftrightarrow v\leq w)$
\end{description}
\end{enumerate}
\end{table}


For a frame $\mathfrak{F}$ of {\bf SRF}$^*_\tau$, the following result ensures that its lattice of clopen elements is an object of {\bf NLE}$_\tau$.

\begin{prop}\rm\label{restriction to clopen prop}
Let $\mathfrak{F}=(X,I,Y,(R_\sigma)_{\sigma\in\tau})$ be a sorted residuated frame in the category {\bf SRF}$^*_\tau$, with a relation $R_\sigma$ of sort $\sigma=(i_{n+1};\vec{i_j})$ for each $\sigma$ in $\tau$. 
Then the clopen elements in $\gpsi$ form a lattice and
$\overline{\alpha}^1_{R_\sigma}$  restricts to a normal operator  of distribution type $\delta=(\vec{i_j};i_{n+1})$ on the lattice of clopens.
\end{prop}
\begin{proof}
By the fact that the Galois connection restricts to a duality on clopens, the clopen elements of $\gpsi$ and $\gphi$ are  lattices, given axiom FAx5.
By definition and axiom FAx2*, $\overline{\alpha}_{R_\sigma}(\vec{F})=\bigvee_{\vec{u}\in\vec{F}}(R_\sigma\vec{u})^{\prime\prime}=\bigvee_{\vec{u}\in\vec{F}}R_\sigma\vec{u}= \bigvee_{\vec{u}\in\vec{F}}\Gamma w$, for some $w$ depending on $\vec{u}$ (which is unique, by the separation axiom FAx1).  In particular, $\overline{\alpha}_{R_\sigma}(\Gamma w_1,\ldots,\Gamma w_n)=\bigvee_{\vec{w}\leq\vec{u}}R_\sigma\vec{u}$. By axiom FAx3, if $\vec{w}\leq\vec{u}$, then $R_\sigma\vec{u}\subseteq R_\sigma\vec{w}$, hence $\overline{\alpha}_{R_\sigma}(\Gamma w_1,\ldots,\Gamma w_n)=R_\sigma\vec{w}$, a closed element, by axiom FAx2*. Again by axiom FAx2*, if all sets $\Gamma w_j$ are clopen elements, then also $R_\sigma\vec{w}$ is a clopen element. Hence $\overline{\alpha}_{R_\sigma}$ restricts to an operator on clopen elements. By Theorem \ref{distribution from stability thm}, using axiom FAx4, $\overline{\alpha}_{R_\sigma}$ (which is a sorted operator, of sort $(\vec{i_j};i_{n+1})$), distributes over arbitrary joins in each argument place. It follows that the single-sorted operator $\overline{\alpha}^1_{R_\sigma}$ obtained by appropriate composition with the Galois connection is a normal operator of distribution type $\delta=(\vec{i_j};i_{n+1})$ on the lattice of clopen elements in $\gpsi$.
\end{proof}

Propositions \ref{restriction to clopen prop} and the next Proposition verify that $\type{L}^*:{\bf SRF}^*_\tau\lra{\bf NLE}_\tau$ is a well defined functor.

\begin{prop}\rm
\label{star is functor}
Let $\type{L}^*$ be defined so that for a sorted residuated frame $\mathfrak{F}$ with relations in {\bf SRF}$^*_\tau$, $\type{L}^*(\mathfrak{F})$ is the normal lattice expansion of its stable clopen elements (the clopens of $\gpsi$).
If $\pi=(p,q):\mathfrak{F}_2\lra\mathfrak{F}_1$ is a morphism in {\bf SRF}$^*_\tau$, then $\type{L}^*(\pi)=\pi^{-1}$ is a homomorphism of normal lattice expansions  from clopens in $\mathcal{G}(X_1)$ to clopens of $\mathcal{G}(X_2)$.
\end{prop}
\begin{proof}
Axioms MAx5 and MAx6 ensure that  $\type{L}^*(\pi)=\pi^{-1}$ maps clopens to clopens, hence it restricts to a homomorphism of the normal lattice expansions of clopens.
\end{proof}

We next verify that the functor $\type{F}:{\bf NLE}_\tau\lra{\bf SRF}_\tau$ is in fact a functor $\type{F}:{\bf NLE}_\tau\lra{\bf SRF}^*_\tau$.

\begin{prop}\rm
The canonical frame of a normal lattice expansion is a sorted residuated frame in the category ${\bf SRF}^*_\tau$.
\end{prop}
\begin{proof}
Axioms FAx1-FAx3 were verified in Lemma \ref{elementary props in canonical}. For the strengthened axiom FAx2$^*$, the proof follows from the fact, proven in \cite{dloa}, Lemma 6.7, that the point operators $\widehat{f}$ map principal filters/ideals to principal filters/ideals. Lemma \ref{canonical frame properties prop} verified axiom FAx4. Clopens in the canonical frame are the sets $\Gamma x_a$, for a principal filter $x_a$ (similarly for ideals) and $\Gamma x_a\cap\Gamma x_b=\Gamma x_{a\vee b}$, so axiom FAx5 holds. By join-density of principal filters (similarly for ideals) $x=\bigvee_{a\in x}x_a$ and then $\Gamma x=\Gamma(\bigvee_{a\in x}x_a)=\bigcap_{a\in x}\Gamma x_a$, hence axiom FAx6 is true for the canonical frame. Clopen sets $\Gamma x_a$ are precisely the sets in the image of the representation map $h(a)=\{x\in\filt(\mathcal{L})\midsp a\in x\}$ and it is by a standard argument in Stone duality that the topology generated by the subbasis $\mathcal{S}=\{h(a)\midsp a\in\mathcal{L}\}\cup\{-h(a)\midsp a\in\mathcal{L}\}$ is compact and totally separated and the compact-open sets (clopen, since the space is totally separated, hence Hausdorff) are precisely the sets $h(a)$, for a lattice element $a$. Proof details can be found in \cite{sdl}, Lemma 2.5, and thereby axiom FAx7 holds for the canonical frame as well, which is then an object of the category {\bf SRF}$^*_\tau$, as claimed.
\end{proof}

\begin{prop}\rm
Let $h:\mathcal{L}\lra\mathcal{L}^*$ be a morphism in the category {\bf NLE}$_\tau$ and $\pi=(p,q):\mathcal{L}^*_+\lra\mathcal{L}_+$ be the canonical {\bf SRF}$_\tau$ morphism, $p=h^{-1}:X^*\lra X$ and $q=h^{-1}:Y^*\lra Y$. Then for a filter $u\in X$, $\pi^{-1}(\Gamma u)$ is a closed element in $\mathcal{G}(X^*)$. Similarly for ideals.
\end{prop}
\begin{proof}
Let $w_h=\langle h[u]\rangle$ be the filter generated by the set $h[u]=\{h(a)\midsp a\in u\}$. By calculating $\pi^{-1}(\Gamma u)$, it is easily seen that $\pi^{-1}(\Gamma u)=\Gamma w_h$. Similarly for ideals.
\end{proof}

By the above arguments, $\type{L}^*,\type{F}$ are contravariant functors on our categories of interest $\type{F}:{\bf NLE}_\tau\leftrightarrows({\bf SRF}^*_\tau)^{op}$.

\begin{thm}[Stone Duality]\rm
$\mathcal{L}\iso\type{L}^*\type{F}(\mathcal{L})$ and $\mathfrak{F}\iso\type{F}\type{L}^*(\mathfrak{F})$.
\end{thm}
\begin{proof}
There is nothing new to prove here and 
details for the proof of the isomorphisms can be found in \cite{sdl} already for lattices and in \cite{sdl-exp} for normal lattice expansions.
\end{proof}

\section{Concluding Remarks and Further Research}
\label{conc}
We revisited in this article the question of Stone duality for lattices with quasioperators (normal lattice expansions, in our terminology), first addressed in \cite{sdl-exp}.  We improved on the results of \cite{sdl-exp} in the following sense.

First, the category of frames is specified with a simpler axiomatization on relations, to ensure that the induced operators are completely normal lattice operators. Gehrke's notion of stability of sections \cite{mai-gen} was used and, by introducing a notion of sorted conjugate operators we argued that the induced sorted operators, for any sort type $\sigma=(i_{n+1};\vec{i_j})$ of the relations, distribute over arbitrary joins of Galois sets (stable, or co-stable, according to the sort type of the relation).  By composition with the Galois connection, completely normal single-sorted operators of distribution type $\delta=(\vec{i_j};i_{n+1})$ are obtained. The stability requirement implies complete distribution, but as far as we can see the two are not equivalent.

Second, frame morphisms in \cite{sdl-exp} were tailored to the need to prove a Stone duality and were thus keyed only to the requirement that their inverses preserve clopens. Based on Goldblatt's recent notion of bounded morphisms for polarities \cite{goldblatt-morphisms2019} we defined weak bounded morphisms for polarities (equivalently, sorted residuated frames) and extended to the case of frames with relations. The extended definition for morphisms is different and simpler than Goldblatt's. The extended notion of a frame morphism was shown to satisfy the requirement that its inverse is a homomorphism of the complex algebras of the frames. To ensure that a Stone duality result is provable, we strengthened the axiomatization of frame relations and morphisms in Section \ref{stone section}.

Third, we expanded on our results in \cite{redm} by showing that completely normal operators on the stable sets lattice of a frame are obtained by taking the closure of the restriction of classical, sorted image operators to Galois sets. This provides a proof, at the representation level, that the logics of normal lattice expansions are fragments of corresponding sorted residuated polymodal logics (their modal companions).

This latter opens up some new problems to investigate, given the results established in this article. Essentially, the research direction opened is one of reducing problems on non-distributive logics (via translation to their modal companions) to problems on sorted residuated polymodal logics. We leave these issues for further research (initiated in \cite{pll7,redm}).

\bibliographystyle{plain}


\end{document}